\documentclass[a4paper]{amsart}
\usepackage{amsmath,amsthm,amssymb,latexsym,epic,bbm,comment}
\usepackage{graphicx,enumerate,stmaryrd,tabularx}
\usepackage[all,2cell]{xy}
\xyoption{2cell}

\newtheorem{theorem}{Theorem}
\newtheorem{lemma}[theorem]{Lemma}
\newtheorem{corollary}[theorem]{Corollary}
\newtheorem{proposition}[theorem]{Proposition}

\newtheorem{example}[theorem]{Example}

\newtheorem{remark}[theorem]{Remark}
\usepackage[all]{xy}
\usepackage[active]{srcltx}
\usepackage[parfill]{parskip}
\usepackage{enumerate}

\usepackage{ltablex}

\newcommand{\tto}{\twoheadrightarrow}
\newcommand{\shift}[1]{  \langle #1 \rangle}

\usepackage{tikz}
\usetikzlibrary{fit,matrix,positioning,calc}
\usepackage{tkz-graph}

\begin{document}
\title[BGG complexes in singular blocks of category ${\mathcal{O}}$]
{BGG complexes in singular blocks of category ${\mathcal{O}}$}

\author{Volodymyr Mazorchuk and Rafael Mr{\dj}en}

\begin{abstract}
Using translation from the regular block, we construct and analyze properties of 
BGG complexes in singular blocks of BGG category ${\mathcal{O}}$. We provide
criteria, in terms of the Kazhdan-Lusztig-Vogan polynomials, for such complexes 
to be exact. In the Koszul dual picture, exactness of BGG complexes is expressed 
as a certain condition on a generalized Verma flag of an indecomposable projective 
object in the corresponding block of parabolic category ${\mathcal{O}}$. 

In the second part of the paper, we construct BGG complexes in a more general setting 
of balanced quasi-hereditary algebras and show how our results for singular blocks
can be used to construct BGG resolutions of simple modules in ${\mathcal{S}}$-subcategories 
in ${\mathcal{O}}$.
\end{abstract}

\maketitle

\section{Introduction and description of the results}\label{s1}

The classical BGG resolution from \cite{bernstein1975differential} 
is a resolution of a 
simple finite dimensional module over a semi-simple complex finite 
dimensional Lie algebra
in terms of Verma modules. It has many applications, for example, 
it can be used to compute
all self-extensions of this simple finite dimensional module inside 
BGG category $\mathcal{O}$,
see \cite{BGG2}. Various generalizations and analogues of BGG 
resolution were studied in many 
different contexts, see, e.g.,  
\cite{lepowsky1977generalization,rocha-caridi1980splitting,
boe2009kostant,enright2004resolutions,futorny1998bgg,konig,kevin}.
Some geometric constructions of BGG complexes in the setting of 
homogeneous bundles and invariant differential operators (which is, 
in a certain sense, dual to category $\mathcal{O}$) were described in
\cite{cap2001bernstein,pandzic2016bgg,mrden2017singular,husadzic2018singular}.

One natural generalization is to consider resolutions of finite dimensional
simple modules using ``standard'' modules in other categories, for example
in the parabolic version of $\mathcal{O}$. This question was studied 
in \cite{lepowsky1977generalization}. Another natural generalization is
to ask which other simple module in $\mathcal{O}$ have resolutions by
Verma modules. For modules inside the regular block of $\mathcal{O}$,
this question was studied in \cite{boe2009kostant}.

The aim of the present paper is to investigate which simple modules in
singular blocks of $\mathcal{O}$ have resolutions by Verma modules.
We give a combinatorial answer involving M{\"o}bius function for the
poset of shortest coset representatives in the cosets of a Weyl group
modulo the parabolic subgroup associated with the singularity and 
Kazhdan-Lusztig-Vogan polynomials. The answer is explicit enough to
be verifiable by a computer, so we provide the lists of such modules
in small ranks (or, more precisely, we provide lists of modules which
do not have this property as in small ranks the number of the latter
modules is significantly smaller).

Our results also have applications to construction of BGG type resolutions
for non-quasi-hereditary generalizations of category $\mathcal{O}$
studied in \cite{futorny2000s}. In fact, our results provide a complete
answer for existence of BGG resolutions of simple modules in the regular
block for such categories. Using the equivalence in \cite{MS}, this
implies existence of BGG-type resolutions for very general setups of
parabolically induced modules. 

The paper is organized as follows. We describe our basic setup in
Section~\ref{s2}. In Section~\ref{s3}, we collect some auxiliary
statement about combinatorics of Bruhat order on Weyl groups.
In Section~\ref{s4} we define singular BGG complexes and we
study their exactness, in terms of Kazhdan-Lusztig-Vogan polynomials,
in Section~\ref{s5}. In Section~\ref{s6} we connect the problem
of existence of BGG resolutions with the Koszul dual picture
of parabolic category $\mathcal{O}$, where
the problem is reformulated in terms of Verma flags of 
indecomposable projective modules. This allows us to give a
sufficient and necessary conditions for existence of BGG resolution
in  Proposition~\ref{item:kostantprop}. 
In Section~\ref{s7} we list results of computations in low rank case.
In Section~\ref{s8} we interpret BGG complexes and resolution in terms
of complexes of structural modules over balanced quasi-hereditary algebras.
Finally, in Section~\ref{s9} we describe how our results can be applied
to construct BGG resolutions for $\mathcal{S}$-subcategories in $\mathcal{O}$.
\vspace{0.5cm}
 
\textbf{Acknowledgements:} This research was partially supported by
the Swedish Research Council, G{\"o}ran Gustafsson Stiftelse and Vergstiftelsen. 
R. M. was also partially supported by the QuantiXLie Center of Excellence grant 
no. KK.01.1.1.01.0004 funded by the European Regional Development Fund.
We thank Axel Hultman for help with Lemma~\ref{lemma:intersections_minimal}.

\section{Setup}\label{s2}

In this paper we work over the base filed $\mathbb{C}$ of complex numbers.

We let $\mathfrak{g}$ denote a semi-simple finite dimensional Lie algebra
with a fixed triangular decomposition ${\mathfrak{g}} = {\mathfrak{n}}^- \oplus {\mathfrak{h}} \oplus {\mathfrak{n}}^+$. Associated to such a  triangular decomposition, we have the
corresponding BGG category $\mathcal{O}$, see \cite{BGG2,Hu}.

For $\lambda\in\mathfrak{h}^*$, we denote by ${\Delta}(\lambda)$ the
Verma module with highest weight $\lambda$. The simple quotient of 
${\Delta}(\lambda)$ is denoted by $L(\lambda)$, and the indecomposable 
projective cover of $L(\lambda)$ in $\mathcal{O}$ is denoted by $P(\lambda)$.

Let $W$ denote the Weyl group of $\mathfrak{g}$ which acts on $\mathfrak{h}^*$
in the usual way. The above triangular decomposition leads to a decomposition
of the root system of $\mathfrak{g}$ into positive and negative roots and
we denote by $\rho$ the half of the sum of all positive roots. Then the {\em dot action}
of $W$ on $\mathfrak{h}^*$ is given by $w \cdot \lambda = w(\lambda + \rho) - \rho$.

Category $\mathcal{O}$ decomposes into {\em blocks} with respect to the action of the
center of the universal enveloping algebra of $\mathfrak{g}$. 
For $\lambda\in\mathfrak{h}^*$, we denote by ${\mathcal{O}}_\lambda$
the block which corresponds to the central character of the Verma module $\Delta(\lambda)$.
Thanks to Soergel's combinatorial description of blocks of $\mathcal{O}$ from \cite{So2}, 
without loss of generality we may work with {\em integral} weights.

Let $\{\alpha_1, \ldots, \alpha_n \}$ be a fixed (standard) ordering of simple roots.
The simple reflection in $W$ corresponding to $\alpha_i$ is denoted $s_{\alpha_i} = s_i$.

A weight $\lambda$ is dominant if $\langle \lambda + \rho, \alpha \rangle \geq 0$ for all simple roots $\alpha$. It is regular if $\langle \lambda + \rho, \alpha \rangle \neq 0$ for all simple roots $\alpha$, otherwise it is singular.

For a subset $S\subset \{\alpha_1, \ldots, \alpha_n \}$, we have the corresponding
parabolic subcategory $\mathcal{O}^S$ in $\mathcal{O}$ as defined in \cite{rocha-caridi1980splitting}.
If $\lambda\in\mathfrak{h}^*$ is dominant and such that its stabilizer in $W$ 
with respect to the dot-action is the
parabolic subgroup generated by simple reflections from $S$, we will also use the
notation ${\mathcal{O}}^\lambda:=\mathcal{O}^S$. The generalized Verma module 
in ${\mathcal{O}}^\lambda$ with highest weight $\mu$ is denoted ${\Delta}^{\lambda}(\mu)$
and its indecomposable projective cover in  ${\mathcal{O}}^\lambda$ is denoted $P^{\lambda}(\mu)$.

We refer to \cite{Hu} for details.
\vspace{1cm}

\section{Combinatorics of the Weyl group}\label{s3}

\subsection{Conventions and preliminaries}\label{s3.1}

Consider the usual length function $l$ and the Bruhat order $<$ on $W$. 
Put $u \to v$ if $u<v$ and $l(v)=l(u)+1$. With respect to Bruhat order,
$W$ is a {\em graded poset} with degree given by the length function,  
in particular, $u \leq v$ if and only if there  is a path 
\begin{equation}\label{eq1}
u=x_0 \to x_1\to \ldots \to x_k= v. 
\end{equation}
Recall that $u \leq v$ also means that some (equivalently, any) reduced 
expression for $v$ contains a reduced expression for $u$ as a subword. 
Also recall that, for a simple reflection $s$, $sv<v$ if and only if $v$ 
has a reduced expression that starts with $s$, moreover, in the latter case 
$sv$ does not have such an expression. This will be often used without mentioning. 
For details,  see \cite[Chapter~2]{bjorner2006combinatorics}.

Fix a dominant, but possibly singular weight $\lambda$, and set $S$ to be the set of 
simple roots $\alpha$ for which $\langle \lambda+\rho,\alpha \rangle = 0$. 
Elements of $S$ are called \emph{singular roots}, and reflections $s_\alpha$, 
for $\alpha \in S$, are called \emph{singular reflections}. Singular 
reflections generate a (parabolic) subgroup $W_\lambda \leq W$, which is precisely the stabilizer of $\lambda$ with respect to the dot action. Denote by 
$W^\lambda$ the set of {\em minimal length} representatives of the cosets 
$W/W_\lambda$. It is a graded subposet of $W$. For $u, v \in W^\lambda$, 
we have $u \leq v$ if and only if there is a path \eqref{eq1} 
completely contained in $W^\lambda$. Recall the following standard fact.

\begin{lemma}[Kostant, {\cite[Proposition 5.13.]{kostant1961lie}}]\label{lem1}
Any $v \in W$ can be uniquely decomposed as $v = v^\lambda v_\lambda$, 
where $v^\lambda\in W^\lambda$ and $v_\lambda \in W_\lambda$. 
Moreover, $l(v)=l(v^\lambda)+l(v_\lambda)$.
\end{lemma}

Denote by $w_0$ the longest element in $W$, and decompose it via Kostant's 
lemma: $w_0 = (w_0)^\lambda (w_0)_\lambda$. Then $(w_0)_\lambda$ is the 
longest element in $W_\lambda$, and we denote it, as usual, by $w_0^\lambda$. 
The set $W^\lambda w_0^\lambda$ is precisely the set of the {\em longest 
representatives} of the cosets $W/W_\lambda$. Since it will be frequently 
used, we reserve a special notation $\tilde{W}^\lambda := W^\lambda w_0^\lambda$ for it.

\subsection{Intersection of intervals and cosets}\label{s3.2}

Here we prove the main auxiliary combinatorial statements 
that will be used in the next section.

\begin{lemma}\label{lemma:cosets_iso}
Assume $x \in W^\lambda$ and $s$ is a simple reflection.
\begin{enumerate}[$($a$)$]
\item\label{lemma:cosets_iso.1} If $sx \not \in xW_\lambda$, 
then $sx \in W^\lambda$ and the multiplication by $s$ from 
the left gives a directed graph isomorphism 
$x W_\lambda \leftrightarrow sx W_\lambda$.
\item\label{lemma:cosets_iso.2} If $sx \in xW_\lambda$, then 
$sz \in xW_\lambda$, for all $z \in xW_\lambda$.
\end{enumerate}
\end{lemma}

\begin{proof}
Let us start with claim~\eqref{lemma:cosets_iso.1}.
Suppose first that $sx<x$. To see that $sx \in W^\lambda$, 
decompose $sx$ as $v^\lambda v_\lambda$ according to Kostant's 
lemma. Then $x=s v^\lambda v_\lambda$, without any cancellation. 
If $v_\lambda$ were non-trivial, this would give a reduced 
expression for $x$ ending in a singular reflection,  which 
contradicts $x \in W^\lambda$. So $v_\lambda=e$, 
and $sx=v^\lambda \in W^\lambda$.

If $x < sx$ and $sx \not \in xW_\lambda$, decompose again 
$sx = v^\lambda v_\lambda$ according to Kostant's lemma. 
Then $x < v^\lambda v_\lambda$, but since $x$ cannot have 
reduced expression ending with a singular root, we must have 
$x < v^\lambda$. We have
\begin{displaymath}
l(x) < l(v^\lambda) \leq l(v^\lambda) + l(v_\lambda) = l(sx) = l(x)+1,  
\end{displaymath}
so $v_\lambda=e$, and $sx = v^\lambda \in W^\lambda$.
Claim~\eqref{lemma:cosets_iso.1} now follows directly from Kostant's lemma.

We proceed with claim~\eqref{lemma:cosets_iso.2}.
If $sx \in xW_\lambda$, then, by Kostant's lemma, $sx=xt$, for some simple singular reflection $t$. Any $z \in xW_\lambda$ has the form $z=xu$ for some $u \in W_\lambda$, so $sz = sxu = xtu \in x W_\lambda$. This completes the proof.
\end{proof}

The next statement was suggested to us by Axel Hultman who also provided
some hints about the proof.

\begin{lemma}\label{lemma:intersections_minimal}
Assume that $x \in W^\lambda$ and $w \in W$ are such that $x \leq w$. 
Then the intersection $[e,w] \cap xW_\lambda$ has a unique maximal element.
\end{lemma}

\begin{proof}
We prove this by induction on $l(w)$. The basis of the induction 
$w=x=e$ is trivial. Let $w$ and $x$ be as in the statement, 
and choose a simple reflection $s$ such that $sw < w$. Recall 
that in this case $w$ has a reduced expression starting 
with $s$, that is $w=s s_2 \ldots s_{l(w)}$, and thus we also have
$sw=s_2 \ldots s_{l(w)}$.

By induction, we can do the following:
\begin{itemize}
\item If $sx \leq sw$, we set $y'$ to be the unique maximal 
element in $[e,sw] \cap sxW_\lambda$.
\item If $x \leq sw$, we set $y''$ to be the unique maximal 
element in $[e,sw] \cap xW_\lambda$.
\end{itemize}
Using $s$, $y'$ and $y''$, we will construct the unique maximal element $y$ in $[e,w] \cap xW_\lambda$. We have to distinguish between several cases.

{\bf Case~1.} Assume that $sx < x$. Obviously, $sx \not \in xW_\lambda$, so,
by Lemma~\ref{lemma:cosets_iso}, we have $sx \in W^\lambda$ and
thus there is a directed graph isomorphism $sx W_\lambda \to x W_\lambda$.

From $sx < x \leq w$, it follows that $sx \leq sw$ and thus $y'$ exists. 
We have $sy' \in xW_\lambda$ and, since $y' \leq sw$, we have $sy' \leq w$. 
So $y:=sy'$ is in the intersection $[e,w] \cap xW_\lambda$. 
We want to see that $y$ is the unique maximum element in this intersection.

Take any $z$ in $[e,w] \cap xW_\lambda$. Then $sz \in sx W_\lambda$ and,
since $sz<z \leq w$, we have $sz \leq sw$. By the inductive assumption, 
$sz \leq y'$ and therefore $z \leq sy'=y$. This completes Case~1.

{\bf Case~2.} Assume that $x < sx$. No reduced expression of $x$ can start with $s$, 
so, in this case, we have $x \leq sw$. Therefore, $y''$ exists by the 
inductive assumption. Note that $y'' \leq w$.

{\bf Subcase~2a.} Assume that $sx \not \in xW_\lambda$. By 
Lemma~\ref{lemma:cosets_iso}, we have a directed graph isomorphism 
$x W_\lambda \to sx W_\lambda$. In particular, no element in 
$x W_\lambda$ can have a reduced expression that starts with $s$. 
Consequently, if $z \in xW_\lambda$ with $z \leq w$, then $z \leq sw$. 
Therefore $z \leq y''$, and thus $y:=y''$ is the element we are looking for.

{\bf Subcase~2b.} Assume that $sx \in xW_\lambda$. By 
Lemma~\ref{lemma:cosets_iso}, we have $sy'' \in xW_\lambda$. Note that 
$sy'' \leq w$. In this case we define $y := \max \{y'',sy''\}$. 
Take $z \in xW_\lambda$ with $z \leq w$. Then also $sz \in xW_\lambda$. 
We either have $z \leq sw$ or $sz \leq sw$. By induction, we either  have
$z \leq y'' \leq y$ or $sz \leq y''$. In the former case we are done. 
In the latter case there are four possibilities:
\begin{align*}
&y'' < sy'' \text{ and } z < sz, &  &y'' < sy'' \text{ and } sz < z,  \\
&sy'' < y'' \text{ and }z < sz,        &   \text{or } \ & sy'' < y'' \text{ and } sz < z. 
\end{align*}
For each of them, it is straightforward to check  that $z \leq y$.
\end{proof}

\begin{remark}
The special case $x=e$ of the previous lemma is proved in 
\cite[Lemma~7]{hombergh1974about}.
\end{remark}

\begin{lemma}
\label{lemma:intersections_maximal}
Assume that $x \in W^\lambda$ and $w \in W$ are such that 
$w \leq x$. Then the intersection $[w,w_0] \cap xW_\lambda$ has a 
unique minimal element. Moreover, the intersection $[w,w_0] \cap xW_\lambda$
is isomorphic, as a directed graph, to the interval $[y,w_0^\lambda]$ in $W_\lambda$. 
\end{lemma}

\begin{proof}
The first claim follows from Lemma~\ref{lemma:intersections_minimal} 
by multiplication with $w_0$. The second claim follows from Kostant's lemma. 
Indeed, $y$ can be taken as the $W_\lambda$-component of the unique 
minimum from the intersection.
\end{proof}

\begin{lemma}\label{lemma:partition_of_intersection}
Assume that $x \in W^\lambda$ and $w \in W$ are such that $w \leq x$.
Assume further that the intersection $[w,w_0] \cap xW_\lambda$ is not 
a singleton. Then there is a partition ${\mathcal{P}}$ of 
$[w,w_0] \cap xW_\lambda$ consisting of $2$-subsets such that:
\begin{enumerate}[$($a$)$]
\item \label{item:partition1}
For any $\{z,z' \} \in {\mathcal{P}}$, there is an arrow $z \to z'$ or $z' \to z$.
\item \label{item:partition2}
If both $\{z \to z' \}$ and $\{t \to t' \}$ are in ${\mathcal{P}}$ 
and $l(z) = l(t)$ but $z \neq t$, then $z \not\leq t'$.
\end{enumerate}
\end{lemma}

\begin{proof}
Because of the second claim in Lemma~\ref{lemma:intersections_maximal}, 
it is enough to prove the statement for any interval $[y,w_0] \subseteq W$, 
where $y < w_0$.

Take any simple reflection $s$ such that $y<sy$. The existence of such 
$s$ is guaranteed by \cite[Proposition~2.3.1]{bjorner2006combinatorics}. 
We claim that the wanted partition can be given by $\{z,sz\}$, 
where $z \in [y,w_0]$.

To see that this partition has property~\eqref{item:partition1} we need
to check the following: if $y < z$ and $sz<z$, then 
$y \leq sz$. But this follows directly from the subword property.

To see that this partition has property~\eqref{item:partition2}, 
suppose $z < sz$, $t < st$, and  $z \leq st$. Then, by the subword property, 
we have $z \leq t$. But, if, in addition, we suppose that $l(z)=l(t)$, 
we obtain $z=t$. The claim follows. 
\end{proof}

\subsection{M\"obius function}\label{s3.3}

Recall that each locally finite partially ordered set has 
its M\"obius function $\mu$, defined as the inverse, in the 
incidence algebra, of the defining $\zeta$-function of the poset. 
For more  information, see \cite[Chapter~3]{stanley2011enumerative}. 
Concretely, the function $\mu$ can be defined, for pairs $w \leq x$, recursively:
\begin{displaymath}
\mu(w,x) = \begin{cases} 1, & \text{if } w=x;  \\ \displaystyle
-\sum_{w \leq z < x} \mu(w,z), & \text{if } w<x.   \end{cases} 
\end{displaymath}
We emphasize that the value $\mu(w,x)$ depends only on the interval $[w,x]$.

The main result of \cite{verma1971mobius} describes the M\"obius function 
for the Bruhat order on Weyl groups. We will need the following generalization 
from \cite[Section~2.7]{bjorner2006combinatorics}, which describes the
M\"obius function for the restriction of the Bruhat order on 
$W^\lambda$, or, equivalently, on $\tilde{W}^\lambda$.

\begin{proposition}\label{proposition:mu_Wlambda}
For $w,x \in \tilde{W}^\lambda$ with $w \leq x$, we have
\begin{displaymath}
\mu^\lambda(w,x) = \begin{cases} 0, &  \text{if there exists } 
z \notin \tilde{W}^\lambda \text{ such that } w < z < x; \\
(-1)^{l(x)-l(w)}, &  \text{otherwise}.   \end{cases}  
\end{displaymath}
In other words, $\mu^\lambda(w,x) =0$ if and only if there 
exists a directed path $w \to \ldots \to x$ that exits $\tilde{W}^\lambda$.
\end{proposition}

\begin{lemma}\label{lemma:mu_singleton}
For $w,x \in \tilde{W}^\lambda$ with $w \leq x$, the 
following assertions are equivalent:
\begin{enumerate}[$($a$)$]
\item\label{item:mu1} The value $\mu^\lambda(w,x) \neq 0$.
\item\label{item:mu2} The intersection $[w,w_0] \cap xW_\lambda$ is a singleton.
\end{enumerate}
\end{lemma}

\begin{proof}
Negation of \eqref{item:mu2} implies $\mu^\lambda(w,x)=0$
by Proposition~\ref{proposition:mu_Wlambda}.

Now, suppose $\mu^\lambda(w,x)=0$, and take $z \notin \tilde{W}^\lambda$ 
such that $w < z < x$. Decompose $z = z^\lambda z_\lambda$ and 
$x = x^\lambda x_\lambda$, according to Kostant's lemma. Then 
$z_\lambda \neq w_0^\lambda$ and $x_\lambda = w_0^\lambda$. By 
\cite[Proposition~2.5.1.]{bjorner2006combinatorics}, we have 
$z^\lambda \leq x^\lambda$. But then, again by Kostant's 
lemma and the subword property, we have 
$z \leq x^\lambda z_\lambda < x$. This means that we have at least two different 
elements, $x^\lambda z_\lambda$ and $x$ in $[w,w_0] \cap xW_\lambda$.
The claim follows.
\end{proof}

\section{Singular BGG complexes}\label{s4}

\subsection{BGG complex}\label{s4.1}

From now on we let $\lambda$ to be a dominant integral weight.
Verma modules and simple modules in the (singular) block 
${\mathcal{O}}_\lambda$ are parameterized by the set $W^\lambda$, 
or, equivalently, by the set $\tilde{W}^\lambda$.
We will use the latter parameterization, since it agrees better with  
translations to walls.

In this section, to each simple module $L(w \cdot \lambda)$ 
in ${\mathcal{O}}_\lambda$, where $w\in \tilde{W}^\lambda$, 
we will attach canonically its (singular) BGG complex. 
So, fix $w \in \tilde{W}^\lambda$. Set
\begin{displaymath}
X^i_w := \left\{ x \in \tilde{W}^\lambda \ \colon \ w \leq x, 
\ \mu^\lambda( w,x ) \neq 0 \text{ and } l(x)= l(w) + i  \right\}. 
\end{displaymath}
Clearly, there exists $t \leq 0$ such that $X^i$ is non-empty,  
for $i = 0,1, \ldots, t$, and empty otherwise. 
Set $X_w = X^0_w \cup \ldots \cup X^t_w$.

Consider the sequence of the form
\begin{equation}\label{equation:singular_BGG}
\ldots \to \bigoplus_{x \in X^{i+1}_w} {\Delta}(x \cdot \lambda) \to \bigoplus_{x \in X^{i}_w} {\Delta}(x \cdot \lambda) \to \ldots \to {\Delta}(w \cdot \lambda) \to L(w \cdot \lambda) \to 0.
\end{equation}
Later on we will define differentials in this sequence such that 
it becomes a complex, which we will call the \emph{(singular) BGG complex} 
attached to $L(w \cdot  \lambda)$. The differentials in 
\eqref{equation:singular_BGG} will  consist of direct sums 
of monomorphisms 
${\Delta}(x' \cdot \lambda) \hookrightarrow {\Delta}(x \cdot \lambda)$, for
$x' \in X^{i+1}_w$ and $x \in X^i_w$ such that there is an arrow
$x \to x'$. These monomorphisms are determined uniquely up to 
a scalar. We will show that it is possible to choose these 
monomorphisms (defined up to scalar) such that 
\eqref{equation:singular_BGG} becomes a complex.
We will also determine under which conditions this complex is exact, 
that is, is a resolution of $L(w \cdot \lambda)$.

\begin{remark}
The necessity of using $X_w$ as the index set of a BGG complex 
follows from the M\"obius inversion formula applied to the 
Euler characteristic of the BGG complex in the graded 
Grothendieck group, see Proposition~\ref{proposition:X_w} for details.
\end{remark}

\subsection{Translation to the wall}\label{s4.2}

The first step in our construction is to consider the classical
BGG resolution of $L(w \cdot 0)$ in ${\mathcal{O}}_0$,
see \cite[Subsection~4.2.]{boe2009kostant}, 
\begin{equation}\label{equation:regular_BGG}
\ldots \to \bigoplus_{\substack{x \geq w \\ l(x)=i+1}} {\Delta}(x \cdot 0) \to \bigoplus_{\substack{x \geq w \\ l(x)=i}} {\Delta}(x \cdot 0) \to \ldots \to {\Delta}(w \cdot 0) \to L(w \cdot 0) \to 0.
\end{equation}
Let us translate the complex in \eqref{equation:regular_BGG},
which we denote by $\boldsymbol{\Delta}$, to the
$\lambda$-wall, that is from ${\mathcal{O}}_0$ to ${\mathcal{O}}_{\lambda}$.

For this we recall the translation functor 
\begin{displaymath}
T_0 ^ \lambda \colon {\mathcal{O}}_0 \to {\mathcal{O}}_\lambda 
\end{displaymath}
which is defined as the unique (up to isomorphism) indecomposable 
projective functor from ${\mathcal{O}}_0$ to ${\mathcal{O}}_\lambda$
that maps ${\Delta}(0)$ to ${\Delta}(\lambda)$, see 
\cite{bernstein1980tensor}. The functor $T_0 ^ \lambda$ is exact, 
has both adjoints, and commutes with simple preserving duality
on $\mathcal{O}$. It acts on Verma modules and simple modules in 
the following way, see \cite[Chapter~2]{jantzen1979moduln}:

\begin{proposition}\label{proposition:translation}
For $w \in W$, we have:
\begin{itemize}
\item $T_0 ^ \lambda {\Delta}(w \cdot 0) = {\Delta}(w \cdot \lambda)$,
\item $T_0 ^ \lambda L(w \cdot 0) =  \begin{cases} L(w \cdot \lambda), 
& \text{if } w \in \tilde{W}^\lambda; \\ 0, & 
\text{otherwise}. \end{cases}$
\end{itemize}
\end{proposition}

Therefore, a homomorphism 
${\Delta}(x' \cdot 0) \hookrightarrow {\Delta}(x \cdot 0)$ 
in $\boldsymbol{\Delta}$ can be mapped either to an 
isomorphism, in case $x'$ and $x$ are in the same $W_\lambda$-coset, 
or, otherwise, to a non-surjective monomorphism 
${\Delta}(x' \cdot \lambda) \hookrightarrow {\Delta}(x \cdot \lambda)$. 
An isomorphism between Verma modules is necessarily 
a non-zero scalar times the identity. So, for all 
$x\in\tilde{W}^\lambda$ with $x \geq w$, all the maps in the 
part of $\boldsymbol{\Delta}$ indexed by $[w,w_0] \cap xW_\lambda$ are 
mapped to non-zero multiples of the identity. We will write 
equalities for such maps in the diagrammatic presentation of 
$T_0^\lambda(\boldsymbol{\Delta})$. 

\begin{lemma}\label{lemn123}
For any $x \in [w,w_0]$, the module ${\Delta}(x \cdot \lambda)$ is 
neither a source, nor a sink, of an equality in the complex 
$T_0^\lambda(\boldsymbol{\Delta})$ if and only if $x \in X_w$.
\end{lemma}

\begin{proof}
Fix $x \in [w,w_0]$, and denote by $\Delta_x$ the 
piece of $\boldsymbol{\Delta}$ indexed by $[w,w_0] \cap xW_\lambda$. 
All homomorphisms in $\Delta_x$ are mapped to isomorphisms 
in $T_0^\lambda(\Delta_x)$, and this piece is connected 
as a graph by Lemma~\ref{lemma:intersections_maximal}. 
Moreover, there are no isomorphisms in $T_0^\lambda(\boldsymbol{\Delta})$ 
that enter or exit $T_0^\lambda(\Delta_x)$. The claim 
now follows from Lemma~\ref{lemma:mu_singleton} and the definition of $X_w$.
\end{proof}

\begin{example}
In type $A_3$, take $\lambda + \rho = (2,1,1,0)$. 
Then $S=\{\alpha_2\}$. Take $w = s_1 s_2$. 
The complex \eqref{equation:regular_BGG} in 
the regular block is displayed in Figure~\ref{figure:regular_BGG_A3}.

\begin{figure}[ht]
\begin{tikzpicture}[>=latex,line join=bevel,]
\node (node_9) at (204.5bp,102.0bp) [draw,draw=none] {$s_{1}s_{2}s_{3}s_{1}$};
  \node (node_8) at (60.5bp,150.0bp) [draw,draw=none] {$s_{3}s_{1}s_{2}$};
  \node (node_7) at (143.5bp,102.0bp) [draw,draw=none] {$s_{1}s_{2}s_{3}s_{2}$};
  \node (node_6) at (112.5bp,150.0bp) [draw,draw=none] {$s_{1}s_{2}s_{1}$};
  \node (node_5) at (82.5bp,102.0bp) [draw,draw=none] {$s_{3}s_{1}s_{2}s_{1}$};
  \node (node_4) at (21.5bp,102.0bp) [draw,draw=none] {$s_{2}s_{3}s_{1}s_{2}$};
  \node (node_3) at (42.5bp,54.0bp) [draw,draw=none]
{$s_{2}s_{3}s_{1}s_{2}s_{1}$};
  \node (node_2) at (182.5bp,54.0bp) [draw,draw=none]
{$s_{1}s_{2}s_{3}s_{2}s_{1}$};
  \node (node_1) at (112.5bp,54.0bp) [draw,draw=none]
{$s_{1}s_{2}s_{3}s_{1}s_{2}$};
  \node (node_10) at (164.5bp,150.0bp) [draw,draw=none] {$s_{1}s_{2}s_{3}$};
  \node (node_11) at (112.5bp,198.0bp) [draw,draw=none] {$s_{1}s_{2}$};
  \node (node_0) at (112.5bp,6.0bp) [draw,draw=none]
{$s_{1}s_{2}s_{3}s_{1}s_{2}s_{1}$};
  \draw [black,->] (node_9) ..controls (177.85bp,115.91bp) and
(151.52bp,129.64bp)  .. (node_6);
  \draw [black,->] (node_0) ..controls (132.35bp,19.612bp) and
(151.32bp,32.616bp)  .. (node_2);
  \draw [black,->] (node_3) ..controls (36.846bp,66.924bp) and
(32.143bp,77.672bp)  .. (node_4);
  \draw [black,->] (node_2) ..controls (171.71bp,67.285bp) and
(162.31bp,78.853bp)  .. (node_7);
  \draw [black,->] (node_10) ..controls (149.87bp,163.5bp) and
(136.8bp,175.57bp)  .. (node_11);
  \draw [black,->] (node_9) ..controls (193.43bp,115.28bp) and
(183.79bp,126.85bp)  .. (node_10);
  \draw [black,->] (node_0) ..controls (92.649bp,19.612bp) and
(73.685bp,32.616bp)  .. (node_3);
  \draw [black,->] (node_8) ..controls (75.127bp,163.5bp) and
(88.204bp,175.57bp)  .. (node_11);
  \draw [black,->] (node_5) ..controls (76.577bp,114.92bp) and
(71.65bp,125.67bp)  .. (node_8);
  \draw [black,->] (node_3) ..controls (53.571bp,67.285bp) and
(63.21bp,78.853bp)  .. (node_5);
  \draw [black,->] (node_5) ..controls (90.668bp,115.07bp) and
(97.588bp,126.14bp)  .. (node_6);
  \draw [black,->] (node_4) ..controls (47.794bp,115.87bp) and
(73.66bp,129.51bp)  .. (node_6);
  \draw [black,->] (node_1) ..controls (139.15bp,67.906bp) and
(165.48bp,81.639bp)  .. (node_9);
  \draw [black,->] (node_2) ..controls (188.42bp,66.924bp) and
(193.35bp,77.672bp)  .. (node_9);
  \draw [black,->] (node_7) ..controls (119.64bp,115.8bp) and
(96.362bp,129.26bp)  .. (node_8);
  \draw [black,->] (node_0) ..controls (112.5bp,18.707bp) and
(112.5bp,28.976bp)  .. (node_1);
  \draw [black,->] (node_2) ..controls (153.38bp,67.978bp) and
(124.39bp,81.893bp)  .. (node_5);
  \draw [black,->] (node_7) ..controls (149.15bp,114.92bp) and
(153.86bp,125.67bp)  .. (node_10);
  \draw [black,->] (node_1) ..controls (86.206bp,67.87bp) and
(60.34bp,81.513bp)  .. (node_4);
  \draw [black,->] (node_4) ..controls (32.294bp,115.28bp) and
(41.693bp,126.85bp)  .. (node_8);
  \draw [black,->] (node_6) ..controls (112.5bp,162.71bp) and
(112.5bp,172.98bp)  .. (node_11);
  \draw [black,->] (node_1) ..controls (120.99bp,67.14bp) and
(128.24bp,78.378bp)  .. (node_7);
\end{tikzpicture}
\caption{Regular BGG complex for $L(s_1s_2 \cdot 0)$ in $A_3$}
\label{figure:regular_BGG_A3}
\end{figure}
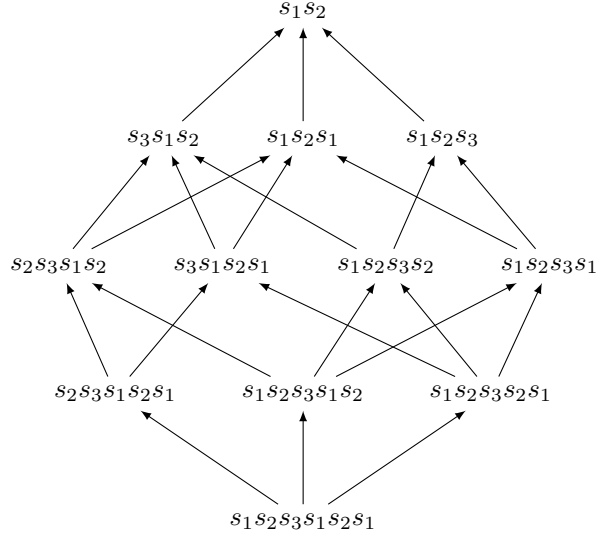

Arrows show the directions of homomorphisms (note that they have the 
opposite directions compared to the arrows in $W$). After applying 
$T_0^\lambda$, we get the complex in ${\mathcal{O}}_\lambda$ 
displayed in Figure~\ref{figure:singular_BGG_A3}.
Elements of $\tilde{W}^\lambda$ are displayed in bold font, and those in $X_w$ are underlined.

\begin{figure}[ht]
\begin{tikzpicture}[>=latex,line join=bevel,]
\node (node_9) at (204.5bp,102.0bp) [draw,draw=none] {$s_{1}s_{2}s_{3}s_{1}$};
  \node (node_8) at (60.5bp,150.0bp) [draw,draw=none] {$\bf \underline{s_{3}s_{1}s_{2}}$};
  \node (node_7) at (143.5bp,102.0bp) [draw,draw=none] {$\bf s_{1}s_{2}s_{3}s_{2}$};
  \node (node_6) at (112.5bp,150.0bp) [draw,draw=none] {$\bf \underline{s_{1}s_{2}s_{1}}$};
  \node (node_5) at (82.5bp,102.0bp) [draw,draw=none] {$\bf \underline{s_{3}s_{1}s_{2}s_{1}}$};
  \node (node_4) at (21.5bp,102.0bp) [draw,draw=none] {$\bf \underline{s_{2}s_{3}s_{1}s_{2}}$};
  \node (node_3) at (42.5bp,54.0bp) [draw,draw=none]
{$\bf \underline{s_{2}s_{3}s_{1}s_{2}s_{1}}$};
  \node (node_2) at (182.5bp,54.0bp) [draw,draw=none]
{$s_{1}s_{2}s_{3}s_{2}s_{1}$};
  \node (node_1) at (112.5bp,54.0bp) [draw,draw=none]
{$\bf s_{1}s_{2}s_{3}s_{1}s_{2}$};
  \node (node_10) at (164.5bp,150.0bp) [draw,draw=none] {$s_{1}s_{2}s_{3}$};
  \node (node_11) at (112.5bp,198.0bp) [draw,draw=none] {$\bf \underline{s_{1}s_{2}}$};
  \node (node_0) at (112.5bp,6.0bp) [draw,draw=none]
{$\bf s_{1}s_{2}s_{3}s_{1}s_{2}s_{1}$};
  \draw [black,->] (node_9) ..controls (177.85bp,115.91bp) and
(151.52bp,129.64bp)  .. (node_6);
  \draw [black,double distance=2pt] (node_0) ..controls (132.35bp,19.612bp) and
(151.32bp,32.616bp)  .. (node_2);
  \draw [black,->] (node_3) ..controls (36.846bp,66.924bp) and
(32.143bp,77.672bp)  .. (node_4);
  \draw [black,->] (node_2) ..controls (171.71bp,67.285bp) and
(162.31bp,78.853bp)  .. (node_7);
  \draw [black,->] (node_10) ..controls (149.87bp,163.5bp) and
(136.8bp,175.57bp)  .. (node_11);
  \draw [black,->] (node_9) ..controls (193.43bp,115.28bp) and
(183.79bp,126.85bp)  .. (node_10);
  \draw [black,->] (node_0) ..controls (92.649bp,19.612bp) and
(73.685bp,32.616bp)  .. (node_3);
  \draw [black,->] (node_8) ..controls (75.127bp,163.5bp) and
(88.204bp,175.57bp)  .. (node_11);
  \draw [black,->] (node_5) ..controls (76.577bp,114.92bp) and
(71.65bp,125.67bp)  .. (node_8);
  \draw [black,->] (node_3) ..controls (53.571bp,67.285bp) and
(63.21bp,78.853bp)  .. (node_5);
  \draw [black,->] (node_5) ..controls (90.668bp,115.07bp) and
(97.588bp,126.14bp)  .. (node_6);
  \draw [black,->] (node_4) ..controls (47.794bp,115.87bp) and
(73.66bp,129.51bp)  .. (node_6);
  \draw [black,double distance=2pt] (node_1) ..controls (139.15bp,67.906bp) and
(165.48bp,81.639bp)  .. (node_9);
  \draw [black,->] (node_2) ..controls (188.42bp,66.924bp) and
(193.35bp,77.672bp)  .. (node_9);
  \draw [black,->] (node_7) ..controls (119.64bp,115.8bp) and
(96.362bp,129.26bp)  .. (node_8);
  \draw [black,->] (node_0) ..controls (112.5bp,18.707bp) and
(112.5bp,28.976bp)  .. (node_1);
  \draw [black,->] (node_2) ..controls (153.38bp,67.978bp) and
(124.39bp,81.893bp)  .. (node_5);
  \draw [black,double distance=2pt] (node_7) ..controls (149.15bp,114.92bp) and
(153.86bp,125.67bp)  .. (node_10);
  \draw [black,->] (node_1) ..controls (86.206bp,67.87bp) and
(60.34bp,81.513bp)  .. (node_4);
  \draw [black,->] (node_4) ..controls (32.294bp,115.28bp) and
(41.693bp,126.85bp)  .. (node_8);
  \draw [black,->] (node_6) ..controls (112.5bp,162.71bp) and
(112.5bp,172.98bp)  .. (node_11);
  \draw [black,->] (node_1) ..controls (120.99bp,67.14bp) and
(128.24bp,78.378bp)  .. (node_7);
\end{tikzpicture}
\caption{$T_0^\lambda$ applied to Figure \ref{figure:regular_BGG_A3}}
\label{figure:singular_BGG_A3}
\end{figure}
\end{example}

\subsection{Cutting off the equalities}\label{s4.3}

The second step is to cut off the equalities that appear in 
$T_0^\lambda(\boldsymbol{\Delta})$, so that only the part indexed by $X_w$ remains.

\begin{lemma}\label{lemma:cutting_off}
If a sequence of modules
\begin{displaymath}
\xymatrix@R=1em{ A \ar[r] & B \ar[r]  \ar[rd] & C 
\ar[r]^f \ar[rd]|-(.25){g} \ar@{}[d]|{\oplus} & D 
\ar[r] \ar@{}[d]|{\oplus} & E \ar[r] & F \\ & & G 
\ar@{=}[r] \ar[ru]|-(.25)h & G  \ar[ru]} 
\end{displaymath}
is a complex (resp. exact), then so is
\begin{displaymath}
\xymatrix{ A \ar[r] & B \ar[r] & C \ar[r]^{f-hg} & D \ar[r] & E \ar[r] & F }. 
\end{displaymath}
\end{lemma}

\begin{proof}
A direct computation.
\end{proof}

Fix $x \in \tilde{W}^\lambda \setminus X_w$ with $x \geq w$. 
We partition the intersection $[w,w_0] \cap xW_\lambda$ into 
a disjoint union of $2$-subsets as in 
Lemma~\ref{lemma:partition_of_intersection}. Denote such a 
partition by ${\mathcal{P}}_x$. We want to apply inductively 
Lemma~\ref{lemma:cutting_off} where $\xymatrix{G \ar@{=}[r] &G}$ 
corresponds to a $2$-subset in ${\mathcal{P}}_x$, so that 
we eventually exhaust all of $[w,w_0] \cap xW_\lambda$. 
Then we repeat the procedure for all such $x$. However,
we have to check that, at each step, the differential of 
our new complex has not changed in an essential way.

\begin{lemma}\label{lemnew1253}
Choose $x \in \tilde{W}^\lambda \setminus X_w$ with $x \geq w$, 
and $\{x_1 \to x_2\} \in {\mathcal{P}}_x$. The following diagrams 
cannot appear as subdiagrams in $T_0^\lambda(\boldsymbol{\Delta})$:
\begin{enumerate}[$($a$)$]
\item\label{lemnew1253.1} for $y,z \in X_w$ such that $y \to z$:
\begin{displaymath}
\xymatrix{ 
{\Delta}(z \cdot \lambda) \ar[r] \ar[rd] & 
{\Delta}(y \cdot \lambda) \\ 
{\Delta}(x_2 \cdot \lambda) \ar@{=}[r] \ar[ru] & {\Delta}(x_1 \cdot \lambda) , 
}
\end{displaymath}
\item\label{lemnew1253.2} for $y \in \tilde{W}^\lambda$, $y \neq x$, $y_1, y_2 \in [w,w_0] \cap y W_\lambda$ such that $y_1 \to y_2$:
\begin{displaymath}
\xymatrix{ 
{\Delta}(y_2 \cdot \lambda) \ar@{=}[r] 
\ar[rd] & {\Delta}(y_1 \cdot \lambda) \\ 
{\Delta}(x_2 \cdot \lambda) \ar@{=}[r] 
\ar[ru] & {\Delta}(x_1 \cdot \lambda) , 
} 
\end{displaymath}
\item\label{lemnew1253.3} for 
$\{ x_1 \to x_2\}, \{ x_3 \to x_4\} \in {\mathcal{P}}_x$ 
with $l(x_1)=l(x_3)$ and $x_1 \neq x_3$:
\begin{displaymath}
\xymatrix{ 
{\Delta}(x_4 \cdot \lambda) \ar@{=}[r] 
\ar@{=}[rd] & {\Delta}(x_3 \cdot \lambda) 
\\ {\Delta}(x_2 \cdot \lambda) \ar@{=}[r] 
\ar@{=}[ru] & {\Delta}(x_1 \cdot \lambda). 
} 
\end{displaymath}
\end{enumerate}
\end{lemma}

\begin{proof}
Consider diagram \eqref{lemnew1253.1}. If there is a non-zero 
morphism ${\Delta}(z \cdot \lambda) \to {\Delta}(x_1 \cdot \lambda)$, 
then $x_1 \leq z$. By \cite[Proposition~2.5.1]{bjorner2006combinatorics},
we have $x \leq z$. But then $\mu^\lambda(w,x) = 0$ implies 
$\mu^\lambda(w,z) = 0$, which is a contradiction with $z \in X_w$.

In diagram \eqref{lemnew1253.2}, the diagonal arrows are injections, 
but not surjections. The composition 
${\Delta}(y_2 \cdot \lambda) \to {\Delta}(x_1 \cdot \lambda) = 
{\Delta}(x_2 \cdot \lambda) \to {\Delta}(y_1 \cdot \lambda)$ 
would map the highest weight space of ${\Delta}(y_2 \cdot \lambda)$ 
into the radical of ${\Delta}(y_1 \cdot \lambda)$. But this 
is impossible since $y_1 \cdot \lambda = y_2 \cdot \lambda$.

The fact that diagram \eqref{lemnew1253.3} is not possible
follows from the choice of the partition ${\mathcal{P}}_x$, 
namely, Lemma~\ref{lemma:partition_of_intersection}\eqref{item:partition2}.
\end{proof}

Lemma~\ref{lemnew1253} has the following consequence:

\begin{corollary}\label{corneq1255}
When using Lemma~\ref{lemma:cutting_off} to cut off 
an equality ${\Delta}(x_2 \cdot \lambda)= {\Delta}(x_1 \cdot \lambda)$ 
corresponding to $\{x_1 \to x_2\} \in {\mathcal{P}}_x$ 
from $T_0^\lambda(\boldsymbol{\Delta})$, the following holds:
\begin{enumerate}[$($a$)$]
\item\label{corneq1255.1} No morphism is changed between 
Verma modules indexed by two elements from $X_w$,
\item\label{corneq1255.2} No isomorphism is changed between 
Verma modules indexed by elements from a different coset $y W_\lambda$,
\item\label{corneq1255.3} No isomorphism is changed 
corresponding to some other pair in the same partition ${\mathcal{P}}_x$.
\end{enumerate}
\end{corollary}
 
Therefore, we can cut-off pairs from ${\mathcal{P}}_x$, for all 
$x \in \tilde{W}^\lambda \setminus X_w$ with $x \geq w$, in any order. 
This implies the following statement.

\begin{theorem}\label{theorem:singular_BGG}
Let $\lambda$ be dominant and integral and $w\in \tilde{W}^\lambda$.
\begin{enumerate}[$($a$)$]
\item \label{theorem:singular_BGG.1}
One can choose non-zero monomorphisms as explained in
Subsection~\ref{s4.1} so that the singular BGG complex 
\eqref{equation:singular_BGG}  becomes a complex.
\item \label{theorem:singular_BGG.2}
If the regular BGG complex \eqref{equation:regular_BGG} of 
$L(w \cdot 0)$ is exact, then so is the singular BGG 
complex \eqref{equation:singular_BGG} of $L(w \cdot \lambda)$.
\end{enumerate}
\end{theorem}

The converse of the second statement above is not true in general, 
since the translation functor can kill homology of the regular 
BGG complex, if it consists of simple modules not parameterized 
by the maximal coset representatives 
(see Proposition~\ref{proposition:translation}).
The smallest example is the following.

\begin{example}
In type $B_3$ take $\lambda + \rho = (1,0,0)$. Then $S=\{\alpha_2, \alpha_3\}$. Take
\begin{displaymath}
w := w_0^\lambda = s_3 s_2 s_3 s_2 \in \tilde{W}^\lambda. 
\end{displaymath}
We want to show that the regular BGG complex of $L(w \cdot 0)$ is 
not exact, but that the singular one corresponding to 
$L(w \cdot \lambda)$ is exact.

Denote by $A$ the set of immediate ascendants of $w$, i.e., 
$A := \{x \in W \colon w \to x \}$. Denote also:
\begin{displaymath}
v := s_2 s_3 s_2 s_1 s_2 s_3 s_2 . 
\end{displaymath}
One can check that the right hand side is a reduced expression of $v$, 
that $v \notin \tilde{W}^\lambda$, and that $w < v$. 
One can calculate the following Kazhdan-Lusztig polynomials:
\begin{equation}
\label{equation:KL_example}
P_{w,x}(q) = \begin{cases} 1+q, & \text{ if }x = v; \\ 
1, & \text{ if }x \geq w, \ x \neq v. \end{cases}
\end{equation} 
From \cite[Theorem 1.4.1]{irving1990singular} it follows that 
the first radical layer ${\operatorname{rad}}^1 {\Delta}(w \cdot 0)$ 
consists of the composition factors $L(x \cdot 0)$ for 
$x \in A \cup \{v\}$. Consider the first part of the regular 
BGG complex \eqref{equation:regular_BGG}:
\begin{displaymath}
\ldots \to \bigoplus_{x \in A} {\Delta}(x \cdot 0) 
\stackrel{d_1}{\longrightarrow} {\Delta}(w \cdot 0) 
\stackrel{d_0}{\twoheadrightarrow} L(w \cdot 0) \to 0.
\end{displaymath}
Clearly ${\operatorname{Ker}} (d_0) = {\operatorname{rad}}\, {\Delta}(w \cdot 0)$. 
The Verma modules on the left hand side map onto the submodules of 
${\Delta}(w \cdot 0)$ with heads 
$L(x \cdot 0) \subseteq {\operatorname{rad}}^1 {\Delta}(w \cdot 0)$ 
with $x \in A$. So 
$L(v \cdot 0) \subseteq {\operatorname{rad}}^1 {\Delta}(w \cdot 0)$ 
is not in the image of $d_1$, and hence the zero-th homology 
group of the regular BGG complex is non-zero since 
${\operatorname{Ker}}(d_0) / \mathrm{Im}(d_1) \tto L(v \cdot 0)$.
In fact, from \eqref{equation:KL_example}, we may deduce 
${\operatorname{Ker}}(d_0) / \mathrm{Im}(d_1) \cong L(v \cdot 0)$.

On the other hand, one can check that $\tilde{W}^\lambda$ is a linear poset:
\begin{displaymath}
\tilde{W}^\lambda = \{ w \to s_1 w \to s_2s_1 w \to s_3s_2s_1 w 
\to s_2s_3s_2s_1 w \to s_1s_2s_3s_2s_1 w\}.
\end{displaymath}
We want to show that $X_w = \{w \to s_1 w \}$. For this, we need 
to find $z \not\in \tilde{W}^\lambda$ such that $w<x<s_2s_1 w$. 
One can check that $z = s_2s_3s_1s_2s_3$ 
does the job. So, we see that our singular BGG complex is:
\begin{equation}\label{eqneq3251}
0 \to {\Delta}(s_1 w \cdot \lambda) \hookrightarrow 
{\Delta}(w \cdot \lambda) \twoheadrightarrow L(w \cdot \lambda) \to 0. 
\end{equation}
From \eqref{equation:KL_example} and 
\cite[Theorem~1.4.2]{irving1990singular} it follows that 
${\operatorname{rad}}^1 {\Delta}(w \cdot \lambda)$ is 
just $L(s_1 w \cdot \lambda)$. This implies that \eqref{eqneq3251} is exact.
\end{example}

In the next section we give a computable 
criteria for a singular BGG complex to be exact.

The singular BGG complex of a dominant simple module admits 
a simpler combinatorial description. Choose a dominant weight $\lambda^c$ such that $\lambda^c + \rho$ is orthogonal precisely to those simple roots $\alpha$ 
which are not orthogonal to $\lambda+\rho$. The subgroup 
$W_{\lambda^c}$ is \emph{complementary} to $W_\lambda$, 
it is generated by all simple non-singular reflections. 

\begin{proposition}\label{proposition:singuar_BGG_dominant}
We have $X_{w_0^\lambda} = W_{\lambda^c} \cdot w_0^\lambda$.
In other words, one can identify weights which appear in the singular BGG complex of $L(w_0^\lambda \cdot \lambda) = L(\lambda)$ with the orbit of $\lambda$ under
the complementary parabolic subgroup.
\end{proposition}

\begin{proof}
Denote by $Y$ the subset of $W^\lambda$ consisting of those $w$ for 
which $[e,w] \subseteq W^\lambda$. It is clear that 
$Y \cdot w_0^\lambda = X_{w_0^\lambda}$, and thus we need to prove 
that $Y = W_{\lambda^c}$. The claim  that $W_{\lambda^c} \subseteq Y$ 
follows from the subword property. For the converse, suppose that 
$w \in W^\lambda \setminus W_{\lambda^c}$. Then, a fixed 
reduced expression of $w$ must contain some singular reflection $s$. 
But then $s < w$ and $s \notin W^\lambda$ and hence $w \notin Y$.
\end{proof}
The description of the parameter set of the singular BGG complex with a 
dominant highest weight given by 
Proposition~\ref{proposition:singuar_BGG_dominant} is related to 
\cite{futorny1998bgg}. See Section \ref{s9}.

\section{Exactness and KLV-polynomials}\label{s5}

\subsection{Singular KLV polynomials}\label{s5.1}

In \cite{irving1990singular}, the following singular version of the KLV (Kazhdan-Lusztig-Vogan) polynomials was introduced. 
Take an integral, possibly singular weight $\mu$ which is \emph{antidominant}, i.e., $w_0 \cdot \mu$ is dominant, and for $y,z$ in $W^\mu$ put
\begin{equation*}
P^\mu_{y,z} (q) = \sum_{j} \dim {\operatorname{Ext}}^j_{\mathcal{O}} 
( {\Delta}(y \cdot \mu) , L(z \cdot \mu) ) \cdot q^\frac{l(z)-l(y)-j}{2}.
\end{equation*}
It is a polynomial, equal to $0$ unless $y \leq z$, equal 
to $1$ if $y=z$, and, in general, of degree at most $(l(z)-l(y)-1)/2$. 

If $\mu$ is regular, then $P^\mu_{y,z} =P_{y,z}$ agrees with the 
usual KL (Kazhdan-Lusztig) polynomials, by a result of Vogan. 
In general, in \cite{soergel1989n-cohomolohy, irving1990singular} 
the following formula that relates the singular KLV-polynomials 
to the usual ones is proved:

\begin{proposition}\label{prop5.1-1}
For $\mu, y, z$ as above, we have
\begin{displaymath}
P^\mu_{y,z} = \sum_{t \in W_\mu} (-1)^{l(t)} P_{yt,z}. 
\end{displaymath}
\end{proposition}

Let us return to our dominant parameterization by $\lambda$. 
If we put $\mu = w_0 \cdot \lambda$, then it is easy to see 
that $W_\mu = w_0 W_\lambda w_0$, and $W^\mu w_0 = \tilde{W}^\lambda$. 
For $w,x \in \tilde{W}^\lambda$ with $w \leq x$ the 
relevant KLV-polynomial is

\begin{equation}\label{equation:dom_KLV}
P^{w_0 \cdot \lambda}_{x w_0,w w_0} (q) = \sum_{j} \dim {\operatorname{Ext}}^j_{\mathcal{O}} ( {\Delta}(x \cdot \lambda) , L(w \cdot \lambda) ) \cdot q^\frac{l(x)-l(w)-j}{2}.
\end{equation}

\begin{theorem}\label{theorem:singular_BGG_exactness}
For $w \in \tilde{W}^\lambda$, the following statements are equivalent:
\begin{enumerate}[$($a$)$]
\item \label{item:exactness1}
The singular BGG complex \eqref{equation:singular_BGG} 
of $L(w \cdot \lambda)$ is exact.
\item \label{item:exactness2}
For all $i \geq 0$, we have
\begin{equation*}
H^i \left( {\mathfrak{n}}^+, L(w \cdot \lambda) \right) = \bigoplus_{x \in X^i_w} {\mathbb{C}}_{x \cdot \lambda},
\end{equation*}
where ${\mathbb{C}}_{x \cdot \lambda}$ is the $1$-dimensional ${\mathfrak{h}}$-module with weight $x \cdot \lambda$.
\item \label{item:exactness3}
For all $i \geq 0$ and $x \in \tilde{W}^\lambda$, we have
\begin{equation*}
\dim {\operatorname{Ext}}^i_{\mathcal{O}} \left({\Delta}(x \cdot \lambda) , L(w \cdot \lambda) \right) = \begin{cases} 1 & \colon x \in X^i_w \\ 0 & \colon \text{otherwise} .\end{cases}
\end{equation*}
\item \label{item:exactness4}
For all $x \in \tilde{W}^\lambda$ with $x \geq w$, we have $P^{w_0 \cdot \lambda}_{x w_0,w w_0} (q) = |\mu^\lambda(w,x)|$.
\end{enumerate}
\end{theorem}

\begin{proof}
Having Theorem \ref{theorem:singular_BGG}, the proof now follows  
\cite[3.4. and 4.3.]{boe2009kostant}. 
We outline it here, for the sake of completeness.

Assume claim~\eqref{item:exactness1}. 
The Killing form induces an ${\mathfrak{h}}$-isomorphism
\begin{displaymath}
H^i \left( {\mathfrak{n}}^+, L(w \cdot \lambda) \right) 
\cong H_i \left( {\mathfrak{n}}^-, L(w \cdot \lambda) \right) , 
\end{displaymath}
and the latter is the $i$-th right derived functor of 
${\mathbb{C}} \otimes_{U({\mathfrak{n}}^-)} -$. It can 
be easily calculated on $L(w \cdot \lambda)$ by using 
\eqref{equation:singular_BGG} as a free $U({\mathfrak{n}}^-)$ 
resolution, and claim~\eqref{item:exactness2} follows.

Assume claim~\eqref{item:exactness2}. Then $L(w \cdot \lambda)$ 
is a generalized Kostant module in the sense of 
\cite{enright2004resolutions}, and their Theorem 2.8. 
shows that \eqref{equation:singular_BGG} must be exact, proving 
claim~\eqref{item:exactness1}.

Claims~\eqref{item:exactness2} and \eqref{item:exactness3} 
are equivalent by \cite[Lemma~5.13]{schmid1981vanishing}, where it is
shown that any $\nu$-weight space of 
$H^i \left( {\mathfrak{n}}^+, L(w \cdot \lambda) \right)$ is, in fact, 
isomorphic to ${\operatorname{Ext}}^i_{\mathcal{O}} 
\left({\Delta}(\nu), L(w \cdot \lambda) \right)$. 

Claims~\eqref{item:exactness3} and \eqref{item:exactness4} are equivalent by \eqref{equation:dom_KLV} and the definition of $X^i_w$.
\end{proof}

\subsection{Kostant modules}\label{s5.2}

In \cite{boe2009kostant}, it is shown that the BGG complex of a 
simple module in the parabolic category ${\mathcal{O}}^\lambda_0$ 
is exact precisely when the simple module is a Kostant module. 
It turns out that this is not true in ${\mathcal{O}}_\lambda$; 
a simple module in ${\mathcal{O}}_\lambda$ whose BGG complex 
is exact does not have to be a Kostant module in the sense of 
\cite[Subsection~3.3]{boe2009kostant}. The problem is that our $X_w$ 
does not have to be an interval in $\tilde{W}^\lambda$. The 
smallest example can be found in type $A_3$. Let 
$\lambda+\rho=(2,1,1,0)$, then $S=\{\alpha_2\}$ and take
$w= s_3 s_1 s_2$. The posets 
$[w,w_0] \supseteq [w,w_0] \cap \tilde{W}^\lambda \supseteq X_w$
are shown in Figure~\ref{figure:BGG_notKostant}. Elements of 
$\tilde{W}^\lambda$ are displayed in bold font, those 
in $X_w$ are underlined, and the arrows follow the Bruhat order.

\begin{figure}[ht]
\begin{tikzpicture}[>=latex,line join=bevel,]
\node (node_7) at (96.0bp,150.0bp) [draw,draw=none] {$\bf \underline{s_{3}s_{1}s_{2}}$};
  \node (node_6) at (157.0bp,102.0bp) [draw,draw=none] {$\bf \underline{s_{3}s_{1}s_{2}s_{1}}$};
  \node (node_5) at (96.0bp,102.0bp) [draw,draw=none] {$\bf \underline{s_{1}s_{2}s_{3}s_{2}}$};
  \node (node_4) at (35.0bp,102.0bp) [draw,draw=none] {$\bf \underline{s_{2}s_{3}s_{1}s_{2}}$};
  \node (node_3) at (96.0bp,54.0bp) [draw,draw=none]
{$\bf \underline{s_{2}s_{3}s_{1}s_{2}s_{1}}$};
  \node (node_2) at (166.0bp,54.0bp) [draw,draw=none]
{$s_{1}s_{2}s_{3}s_{2}s_{1}$};
  \node (node_1) at (26.0bp,54.0bp) [draw,draw=none]
{$\bf \underline{s_{1}s_{2}s_{3}s_{1}s_{2}}$};
  \node (node_0) at (96.0bp,6.0bp) [draw,draw=none]
{$\bf s_{1}s_{2}s_{3}s_{1}s_{2}s_{1}$};
  \draw [black,<-] (node_1) ..controls (28.396bp,66.779bp) and
(30.351bp,77.207bp)  .. (node_4);
  \draw [black,<-] (node_0) ..controls (96.0bp,18.707bp) and
(96.0bp,28.976bp)  .. (node_3);
  \draw [black,<-] (node_1) ..controls (45.851bp,67.612bp) and
(64.815bp,80.616bp)  .. (node_5);
  \draw [black,<-] (node_5) ..controls (96.0bp,114.71bp) and
(96.0bp,124.98bp)  .. (node_7);
  \draw [black,<-] (node_4) ..controls (52.116bp,115.47bp) and
(68.198bp,128.12bp)  .. (node_7);
  \draw [black,<-] (node_3) ..controls (113.12bp,67.468bp) and
(129.2bp,80.123bp)  .. (node_6);
  \draw [black,<-] (node_6) ..controls (139.88bp,115.47bp) and
(123.8bp,128.12bp)  .. (node_7);
  \draw [black,<-] (node_0) ..controls (115.85bp,19.612bp) and
(134.82bp,32.616bp)  .. (node_2);
  \draw [black,<-] (node_3) ..controls (78.884bp,67.468bp) and
(62.802bp,80.123bp)  .. (node_4);
  \draw [black,<-] (node_2) ..controls (163.6bp,66.779bp) and
(161.65bp,77.207bp)  .. (node_6);
  \draw [black,<-] (node_0) ..controls (76.149bp,19.612bp) and
(57.185bp,32.616bp)  .. (node_1);
  \draw [black,<-] (node_2) ..controls (146.15bp,67.612bp) and
(127.18bp,80.616bp)  .. (node_5);
\end{tikzpicture}
\caption{BGG resolution of $L(s_3s_1s_2 \cdot \lambda)$ with $\lambda+\rho = (2,1,1,0)$ in type $A_3$, which is not parameterized by an interval in $\tilde{W}^\lambda$}
\label{figure:BGG_notKostant}
\end{figure}
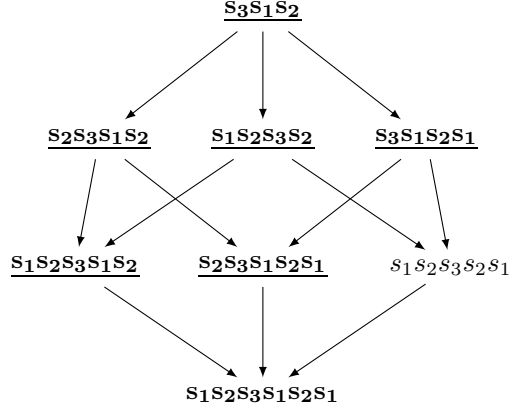 

The phenomenon is related to the fact that certain generalized Verma 
module in the regular parabolic block ${\mathcal{O}}^\lambda$ does 
not have  simple socle. This example suggests that their definition 
of Kostant modules might not be the ``correct one'' for singular blocks.
We will return to this question later in Subsection~\ref{subsection:Kostant}.

\section{Koszul duals of modules with the BGG resolution}\label{s6}

\subsection{Parabolic category ${\mathcal{O}}^\lambda$}\label{s6.1}

A dominant, integral, possibly singular weight $\lambda$ uniquely determines a parabolic subgroup
${\mathfrak{p}} \subseteq {\mathfrak{g}}$ containing 
${\mathfrak{b}}$, which, in turn, defines the parabolic 
category ${\mathcal{O}}^\lambda$, that is, the full subcategory 
of ${\mathcal{O}}$ consisting of ${\mathfrak{p}}$-locally 
finite modules. Note that ${\mathcal{O}}^0 = {\mathcal{O}}$. 
We denote by ${\mathcal{O}}^\lambda_\nu$ the block of 
${\mathcal{O}}^\lambda$ corresponding to the central
character determined by the weight $\nu$. 
In particular, consider the regular block ${\mathcal{O}}^\lambda_0$. 
Generalized Verma modules ${\Delta}^\lambda(w \cdot 0)$, 
and simple modules $L(w \cdot 0)$ in ${\mathcal{O}}^\lambda_0$ 
are parameterized by $w$ running over the set $^\lambda W$ of the minimal 
representatives of the cosets $W_\lambda \backslash W$. It is easy to see that 
the map $w \mapsto w^{-1} w_0$ gives a bijection 
$\tilde{W}^\lambda \stackrel{\sim}{\longrightarrow} {}^\lambda W$. Denote
\begin{displaymath}
\hat{w} := w^{-1}w_0. 
\end{displaymath}

\begin{remark}\label{remark:w0ww0}
The map $w \mapsto \hat{\hat{w}} = w_0 w w_0$ gives rise to an automorphism 
of the Dynkin diagram, and hence of the Coxeter system. It is 
equal to the identity in all simple Coxeter types except $A_n$ and 
odd rank $D_n$, when it is the unique non-identity 
automorphism of the Dynkin diagram. In particular, the map
preserves for KL- and KLV-polynomials, multiplicities and 
${\operatorname{Ext}}$-groups from generalized Verma modules to simple modules.
\end{remark}

\subsection{Graded version of category ${\mathcal{O}}$}
\label{subsection:graded_O}

In what follows, by {\em graded} we mean {\em $\mathbb{Z}$-graded}.

For a block ${\mathcal{O}}^\lambda_\nu$, denote by $E^\lambda_\nu$ 
the endomorphism algebra of a minimal projective generator in 
${\mathcal{O}}^\lambda_\nu$. The category ${\mathcal{O}}^\lambda_\nu$ 
is equivalent to $E^\lambda_\nu\text{-}\mathrm{mod}$, the category of 
finitely generated $E^\lambda_\nu$-modules. In 
\cite{beilinson1996koszul}, it is proved that $E^0_\lambda$ 
and $E^\lambda_0$ have a Koszul grading, and, moreover, 
that they are Koszul dual to each other. This was 
generalized to arbitrary ${\mathcal{O}}^\lambda_\nu$ in 
\cite{backelin1999koszul}, see also \cite{mazorchuk2009applications}.

For a graded algebra $E$, denote denote by $E$-${\operatorname{gmod}}$ 
the category of \emph{graded} finitely generated $E$-modules, 
where for morphisms we take only homogeneous homomorphisms of 
degree zero. All structural modules, that is simple modules, 
(generalized) Verma modules and their duals, indecomposable 
projectives, injectives and tilting modules in both 
${\mathcal{O}}^0_\lambda$ and ${\mathcal{O}}^\lambda_0$ 
admit a graded lift to $E^0_\lambda$-${\operatorname{gmod}}$ 
and $E^\lambda_0$-${\operatorname{gmod}}$, respectively, 
and this lift is unique up to a shift in grading. For 
modules with simple head, the grading is given by their 
radical filtration. In particular, this applies to simple 
modules, generalized Verma modules, and indecomposable 
projectives. We will use the same notation for their graded 
lifts, assuming that their heads have degree $0$. In the 
latter case, their radicals are contained in positive degrees.

For a graded module $\displaystyle M=\bigoplus_{i \in {\mathbb{Z}}} M_i$, 
we denote by $M\shift{k}$ the same module with the 
shifted grading: $M\shift{k}_i = M_{i+k}$. Our BGG complex 
\eqref{equation:singular_BGG} admits a {\em graded lift}, that is
it can be lifted to $E^0_\lambda$-${\operatorname{gmod}}$, and, 
moreover, this lift is \emph{linear} in the sense that the $i$-th term of 
the complex is a direct sum of ${\Delta}(x \cdot \lambda) \shift{i}$, 
where $x$ varies. The analogous statements are true for BGG 
complexes in regular parabolic blocks.

\begin{remark}
It can be shown that any minimal resolution of a simple module 
in ${\mathcal{O}}_\lambda$ by direct sums of generalized Verma 
modules lifts to a linear resolution in 
$E^0_\lambda$-${\operatorname{gmod}}$. This is true as
${\operatorname{Hom}}$-spaces between Verma modules 
are $1$-dimensional and gradable.
\end{remark}

For our further discussion we will use the following statement which is
\cite[Proposition~1.3.1]{beilinson1996koszul}:

\begin{proposition}\label{proposition:BGS_ext_mult}
For $w, x \in \tilde{W}^\lambda$ and for all $i \geq 0$, we have:
\begin{enumerate}[$($a$)$]
\item\label{item:BGS_ext_mult1}
$\dim {\operatorname{Ext}}_{\mathcal{O}}^i 
\left({\Delta}(x \cdot \lambda), L(w \cdot \lambda)\right) = 
\left[ {\operatorname{rad}}^i {\Delta}^\lambda( \hat{x} \cdot 0) 
\colon L(\hat{w} \cdot 0) \right]$,
\item\label{item:BGS_ext_mult2}
$\dim {\operatorname{Ext}}_{\mathcal{O}^\lambda}^i 
\left({\Delta}^\lambda( \hat{x} \cdot 0), L(\hat{w} \cdot 0) \right) 
= \left[ {\operatorname{rad}}^i {\Delta}(x \cdot \lambda) 
\colon L(w \cdot \lambda) \right]$,
\end{enumerate}
where ${\operatorname{rad}}^i$ means the $i$-th layer of the radical filtration.
\end{proposition}

We will also need the following graded version of the BGG reciprocity
(which is proved by the same argument as the classical result):

\begin{proposition}[Graded BGG reciprocity]
\label{proposition:BGG_reciprocity}
For all weights $\nu, \xi$ and all $i \geq 0$, we have
\begin{displaymath}
\left(  P(\xi) \colon {\Delta}(\nu) \shift{i} \right) = 
\left[{\Delta}(\nu) \colon L(\xi) \shift{i} \right] = 
\left[{\operatorname{rad}}^i {\Delta}(\nu) \colon L(\xi) \right] . 
\end{displaymath}
An analogous statement holds in ${\mathcal{O}}^\lambda$.
\end{proposition}

\subsection{(Generalized-)Verma flags of projectives}\label{s6.3}

By Koszul duality, a simple module 
$L(w \cdot \lambda) \in {\mathcal{O}}_\lambda$, where
$w \in \tilde{W}^\lambda$, corresponds to the indecomposable 
projective module $P^\lambda(\hat{w} \cdot 0) \in {\mathcal{O}}^\lambda_0$, 
where $\hat{w}=w^{-1}w_0 \in {}^\lambda W$. 
A simple module $L(\hat{w} \cdot 0) \in {\mathcal{O}}^\lambda_0$ 
corresponds to the indecomposable projective module 
$P(w \cdot 0) \in {\mathcal{O}}_\lambda$. We note that, in the
classical Koszul duality, a simple module 
corresponds to an indecomposable injective module in the Koszul dual picture. 
Therefore, we need to compose the classical Koszul duality with the usual
simple preserving duality on ${\mathcal{O}}$ to get a projective module. 
This is  related to Remark~\ref{remark:w0ww0}.

We want to see how the exactness of the BGG complex of a simple 
module in one category is reflected on its Koszul dual projective object.

\begin{proposition}\label{propnew8615}
For $w \in \tilde{W}^\lambda$, the following statements are equivalent:
\begin{enumerate}[$($a$)$]
\item \label{item:proj_mult_free00}
The BGG complex of $L(\hat{w} \cdot 0) \in {\mathcal{O}}^\lambda_0$ is exact.
\item\label{item:proj_mult_free01}
The projective $P(w \cdot \lambda) \in {\mathcal{O}}_\lambda$ has a multiplicity-free Verma flag.
\item\label{item:proj_mult_free02}
$\left[ {\Delta}(\lambda) \colon L(w \cdot \lambda) \right] = 1$.
\item\label{item:proj_mult_free03}
${\operatorname{End}}(P(w \cdot \lambda))$ is commutative.
\end{enumerate}
\end{proposition}

\begin{proof}
From \cite[Subsection~3.4]{boe2009kostant}, 
Proposition~\ref{proposition:BGS_ext_mult}\eqref{item:BGS_ext_mult2} 
and the BGG reciprocity, it follows that 
claim~\eqref{item:proj_mult_free00} is equivalent to the 
following statement: 
For all $x \in \tilde{W}^\lambda$, $x \leq w$, and $i \geq 0$, we have:
\begin{equation}\label{equation:proj_mult_free00}
\left( P(w \cdot \lambda) \colon {\Delta}(x \cdot \lambda)\shift{i} \right) = 
\begin{cases} 
1, & \text{if }i = l(w)-l(x); \\ 0, & \text{otherwise.}
\end{cases}
\end{equation}
Obviously this implies claim~\eqref{item:proj_mult_free01}.

Assume claim~\eqref{item:proj_mult_free01}, and take 
$x \in \tilde{W}^\lambda$, $x \leq w$ and $i \geq 0$. Recall that 
$L(w \cdot \lambda)$ must always appear in 
${\operatorname{rad}}^{l(w)-l(x)} {\Delta}(x \cdot \lambda)$ 
exactly once (this follows, for example, from 
\cite[Section~1.4.2]{irving1990singular} and the fact that non-zero 
KL-polynomials always have $1$ as the constant term). From this, 
it follows that, if $i=l(w)-l(x)$, then 
$\left( P(w \cdot \lambda) \colon {\Delta}(x \cdot \lambda)\shift{i} \right) = 1$.

So, assume $i \neq l(w)-l(x)$ and 
$\left( P(w \cdot \lambda) \colon {\Delta}(x \cdot \lambda)\shift{i} \right) >0$. 
From the above, we know that 
$\left( P(w \cdot \lambda) 
\colon {\Delta}(x \cdot \lambda)\shift{l(w)-l(x)} \right) = 1$. 
This implies that ${\Delta}(x \cdot \lambda)$ occurs more
than once in $P(w \cdot \lambda)$, which is a contradiction. 
This establishes \eqref{equation:proj_mult_free00}, and hence 
implies claim~\eqref{item:proj_mult_free00}.

Claims~\eqref{item:proj_mult_free01} and \eqref{item:proj_mult_free02}
are equivalent by BGG reciprocity. Claim \eqref{item:proj_mult_free03} 
is equivalent to claim~\eqref{item:proj_mult_free02} by 
\cite[Theorem~7.1]{stroppel2003category_quivers}.
\end{proof}

A similar result in the other direction is weaker, due to the fact 
that projective modules in ${\mathcal{O}}^\lambda_0$ are much more complicated
than projective modules in ${\mathcal{O}}_\lambda$, as the following example suggests. 

\begin{example}
In type $A_3$, take $\lambda+\rho=(1,1,0,0)$ and $w=s_3s_1$. 
Then  $S=\{s_1,s_3\}$. The BGG complex of 
$L(w \cdot \lambda) \in {\mathcal{O}}_\lambda$ is not exact. 
However, its Koszul-dual is the projective module 
$P^\lambda(s_2s_3s_1s_2 \cdot 0) \in {\mathcal{O}}^\lambda_0$, 
which has a multiplicity-free generalized Verma flag consisting of:
\begin{displaymath}
\begin{array}{ccc}
& {\Delta}^\lambda(s_2s_3s_1s_2 \cdot 0)& \\
{\Delta}^\lambda(s_2s_3s_1 \cdot 0)\langle 1\rangle&& 
\ {\Delta}^\lambda(s_2 \cdot 0)\langle 1\rangle \\
& {\Delta}^\lambda(e \cdot 0)\langle 2\rangle.&
\end{array}
\end{displaymath}
Note that the components ${\Delta}^\lambda(s_2s_3s_1s_2 \cdot 0)$ 
and ${\Delta}^\lambda(s_2s_3s_1 \cdot 0)\langle 1\rangle$ are 
``predictable'' in the sense that they correspond to trivial 
KL-polynomials, while the other two, 
${\Delta}^\lambda(s_2 \cdot 0)\langle 1\rangle$ and 
${\Delta}^\lambda(e \cdot 0)\langle 2\rangle$, are coming 
from non-trivial (however, monomial) KL-polynomials.
\end{example}

Recall that a generalized Verma module 
${\Delta}^\lambda(\hat{x} \cdot 0)$, where $\hat{x} \in {}^\lambda W$, 
is the maximal quotient of ${\Delta}(\hat{x} \cdot 0) \in {\mathcal{O}}$ 
that belongs to the category ${\mathcal{O}}^\lambda$. The kernel 
of this quotient consists of all submodules 
${\Delta}(\hat{y} \cdot 0) \subset {\Delta}(\hat{x} \cdot 0)$ 
for which $\hat{y} \notin {}^\lambda W$,
see \cite[Page~187]{Hu}. From this it follows 
that the head of a Verma module ${\Delta}(\hat{w} \cdot 0)$, where
$\hat{w} \in {}^\lambda W$, survives in 
${\operatorname{rad}}^{l(\hat{w})-l(\hat{x})} {\Delta}^\lambda(\hat{x} \cdot 0)$ 
if and only if there is no $\hat{y}$ such that 
$\hat{x} \leq \hat{y} \leq \hat{w}$ and 
$\hat{y} \notin {}^\lambda W$. This is equivalent to 
$\mu^\lambda(w,x) \neq 0$, that is, to $x \in X_w$. We say that this 
particular occurrence of $L(\hat{w} \cdot 0)$ in 
$\Delta^\lambda(\hat{x} \cdot 0)$ is \emph{predictable}. 
In this case, by the BGG reciprocity, we also have
\begin{displaymath}
\left( P^\lambda( \hat{w} \cdot 0) \colon 
{\Delta}^\lambda(\hat{x} \cdot 0)\shift{l(\hat{w})-l(\hat{x})}\right)=1. 
\end{displaymath}
This particular occurrence of ${\Delta}^\lambda(\hat{x} \cdot 0)$ in 
$P^\lambda( \hat{w} \cdot 0)$ is also said to be \emph{predictable}. 
In conclusion, we say that an occurrence of either a simple in a 
Verma, or a Verma in a projective to be predictable, if the shift 
in grading is precisely given by the difference of lengths of the 
parameters. Note that, in ${\mathcal{O}}_\lambda$, an occurrence is 
predictable if and only if the corresponding (ungraded) multiplicity is one.

\begin{remark}
Note the following instance of the Koszul duality. In ${\mathcal{O}}_\lambda$, 
BGG complexes are given by $X_w \subseteq \tilde{W}^\lambda$ which also
coincides with the support of the function
$\mu^\lambda(w,-)$, and the predictable factors are given by the interval 
$[w,w_0] \cap \tilde{W}^\lambda$. In ${\mathcal{O}}^\lambda_0$, 
the roles are reversed (up to the bijection $x \mapsto \hat{x}$).
\end{remark}

\begin{proposition}\label{item:proj_mult_free-st}
For $w \in \tilde{W}^\lambda$, the following statements are equivalent:
\begin{enumerate}[$($a$)$]
\item \label{item:proj_mult_free10}
The singular BGG complex (\ref{equation:singular_BGG}) of $L(w \cdot \lambda) \in {\mathcal{O}}_\lambda$ is exact.
\item \label{item:proj_mult_free11}
The module $P^\lambda(\hat{w} \cdot 0) \in {\mathcal{O}}^\lambda_0$ 
has a generalized Verma flag consisting only of predictable factors. 
\end{enumerate}
\end{proposition}

Note that a generalized Verma flag consisting only of predictable factors
is necessarily a multiplicity free flag.

\begin{proof}
From Theorem \ref{theorem:singular_BGG_exactness}\eqref{item:exactness3}, 
Proposition \ref{proposition:BGS_ext_mult}\eqref{item:BGS_ext_mult1}, 
and the BGG reciprocity it follows that claim~\eqref{item:proj_mult_free10} 
is equivalent to the following statement: 
For all $x \in \tilde{W}^\lambda$, $x \geq w$ and $i \geq 0$, we have:
\begin{equation*}
\left( P^\lambda( \hat{w} \cdot 0) \colon 
{\Delta}^\lambda(\hat{x} \cdot 0) \shift{i}  \right) = 
\begin{cases} 1, & \text{if }x \in X^i_w; \\ 0, 
& \text{otherwise.} \end{cases}
\end{equation*}
This is precisely claim~\eqref{item:proj_mult_free11}.
\end{proof}

Now we can show the necessity of using $X_w$ as the indexing 
set for our singular BGG complexes.

\begin{proposition}\label{proposition:X_w}
Suppose we have a resolution in ${\mathcal{O}}_\lambda$ of the form
\begin{equation*}
\ldots  \to \bigoplus_{x \in Y^{i}}  {\Delta}(x \cdot \lambda)^{\oplus m(x)} 
\to \ldots \to \bigoplus_{x \in Y^0} {\Delta}(x \cdot \lambda)^{\oplus m(x)} 
\to L(w \cdot \lambda) \to 0
\end{equation*}
for some multiplicities $m(x) >0$, which lifts to a linear resolution 
in $E^0_\lambda$-${\operatorname{gmod}}$. Then all $m(x)=1$ and 
$Y^i = X^i_w$, for all $i \geq 0$.
\end{proposition}

\begin{proof}
From linearity and exactness, it follows that, for every 
$x \in Y^i$, we must have $x \geq w$ and $l(x)=l(w)+i$. 
From this it follows that each ${\Delta}(x \cdot \lambda)$, 
where $x \in Y^i$, lifts to 
${\Delta}(x \cdot \lambda)\shift{l(x)-l(w)} 
\in E^0_\lambda$-${\operatorname{gmod}}$.

Consider the Euler characteristic of the lift of the above 
resolution in the graded Grothendieck group:
\begin{equation}\label{equation:Euler_char}
\sum_{x \in [w,w_0]\cap \tilde{W}^\lambda} (-1)^{l(x)-l(w)} m(x) 
[{\Delta}(x \cdot \lambda) \shift{l(x)-l(w)}] = [L(w \cdot \lambda)].
\end{equation}
Here we put $m(x)=0$ if $\displaystyle x \not\in \bigcup_i Y_i$.

Fix $x \in [w,w_0]\cap \tilde{W}^\lambda$. For any
$z \in [w,x]\cap \tilde{W}^\lambda$, the Verma module
${\Delta}(z \cdot \lambda)$ contains the predictable 
occurrence of $L(x \cdot \lambda)\shift{l(x)-l(z)}$. 
This occurrence appears in \eqref{equation:Euler_char} 
as $L(x \cdot \lambda)\shift{l(x)-l(w)}$, independent 
of $z$. A potential non-predictable occurrence can 
appear only in a strictly lower graded component of 
\eqref{equation:Euler_char}. It follows that all 
the predictable occurrences must cancel out, i.e.,
\begin{displaymath}
g(x) := \sum_{z \in [w,x]\cap \tilde{W}^\lambda} 
(-1)^{l(z)-l(w)} m(z) = \begin{cases} 1, &  \text{ if } x=w; 
\\ 0, & \text{ if } x > w. \end{cases} 
\end{displaymath}
By the M\"obius inversion formula,
see e.g. \cite[Proposition 3.7.1.]{stanley2011enumerative},
we have
\begin{align*}
(-1)^{l(x)-l(w)} m(x) &= \sum_{z \in [w,x]\cap
\tilde{W}^\lambda}  g(z) \mu^\lambda(z,x) = 
\mu^\lambda(w,x), \text{ i.e.,} \\
m(x) &= (-1)^{l(x)-l(w)} \mu^\lambda(w,x) =  
\begin{cases} 1, &  \text{if }  x \in X_w; \\ 0, & \text{otherwise}. \end{cases}
\end{align*}
This completes the proof.
\end{proof}
Proposition \ref{proposition:X_w} also follows from Theorem \ref{thmuniquebgg}, but the proof presented here is much more elementary.

\subsection{Kostant modules revisited}\label{subsection:Kostant}

We say that a module $L(w \cdot \lambda)$, where $w \in \tilde{W}^\lambda$, 
is a \emph{Kostant module}, if there is a convex subset 
(which is not necessarily an interval) $Y \subseteq [w,w_0] \cap \tilde{W}^\lambda$ 
such that, for all $i \geq 0$, we have
\begin{displaymath}
H^i({\mathfrak{n}}^+, L(w \cdot \lambda)) = 
\bigoplus_{\substack{x \in Y \\ l(x)=i+l(w)}} {\mathbb{C}}_{x \cdot \lambda}. 
\end{displaymath}

\begin{proposition}\label{item:kostantprop}
For $w \in \tilde{W}^\lambda$, the following statements are equivalent:
\begin{enumerate}[$($a$)$]
\item \label{item:kostant20}
The singular BGG complex (\ref{equation:singular_BGG}) of $L(w \cdot \lambda) \in {\mathcal{O}}_\lambda$ is exact.
\item \label{item:kostant21}
$L(w \cdot \lambda)$ is a Kostant module in the above sense.
\end{enumerate}
\end{proposition}

By Proposition~\ref{item:kostantprop},
the two notions of ``Kostant module'' agree on 
the regular blocks of category ${\mathcal{O}}$.

\begin{proof}
Given claim~\eqref{item:kostant20},
Theorem~\ref{theorem:singular_BGG_exactness}\eqref{item:exactness2} 
implies claim~\eqref{item:kostant21} with $Y=X_w$. '

Assume claim~\eqref{item:kostant21}, for some $Y$, and recall that 
\begin{displaymath}
H^i({\mathfrak{n}}^+, L(w \cdot \lambda))_{x \cdot \lambda} \cong {\operatorname{Ext}}^i_{\mathcal{O}} \left({\Delta}(x \cdot \lambda), L(w \cdot \lambda) \right), 
\end{displaymath}
and that the right hand side has dimension equal to 
$\left( P^\lambda( \hat{w} \cdot 0) \colon {\Delta}^\lambda(\hat{x} \cdot 0) \shift{i}  \right)$. Since $i=l(x)-l(w)$, the dimension is positive if and only if 
the occurrence of ${\Delta}^\lambda(\hat{x} \cdot 0) \shift{i}$
in $P^\lambda( \hat{w} \cdot 0)$ is predictable. Therefore, we obtain $Y=X_w$
implying claim~\eqref{item:kostant20}.
\end{proof}

For general singular-parabolic blocks, the convexness of $Y$ in 
the definition of Kostant modules should probably be expressed in 
terms of the  ``$\mu$-order'' as suggested in 
\cite[Subsection~9.3]{boe2009kostant}, instead of the Bruhat order. 
These two orders agree on ${\mathcal{O}}_\lambda$ and on
${\mathcal{O}}^\lambda_0$.

\section{Non-Kostant modules in low ranks}\label{s7}

In this section, we use 
Theorem~\ref{theorem:singular_BGG_exactness}\eqref{item:exactness4} together with Proposition \ref{prop5.1-1} to list non-Kostant modules, i.e., simple modules whose singular
BGG complex is not exact, in all singular blocks of category 
${\mathcal{O}}$ in ranks up to $4$ for classical Lie algebras, and 
for one large singularity in the exceptional case $F_4$.

In the first column of the following tables, we put singularity sets 
$S$ which determine the blocks ${\mathcal{O}}_\lambda$ (where $\lambda$ 
is a dominant integral weight with singularity $S$). 
In the second column we list all 
$w \in \tilde{W}^\lambda$ such that 
$L(w \cdot \lambda) \in {\mathcal{O}}_\lambda$ is not Kostant.

Instead of $S=\{\alpha_{i_1}, \ldots, \alpha_{i_j} \}$ we will 
just write $S=\{i_1, \ldots, i_j \}$. Similarly, instead of
$w=s_{i_1} \ldots s_{i_j}$ we will write $w=(i_1 \ldots i_j)$. 
All expressions will be reduced. 

The calculations were performed in SageMath, version 8.4.

\subsection{Rank 1 and 2}\label{s7.1}

There are no non-Kostant modules in ranks $1$ and $2$. 
This follows from the fact that all BGG complexes are exact 
already in the regular block, which is a consequence of the 
fact that all $KL$-polynomials are trivial.

\subsection{Rank 3}\label{s7.2}

In type $A_3$, we only present possible singularity sets 
up the unique non-identity  automorphism of the Dynkin diagram
which swaps $1 \leftrightarrow 3$. 
In type $B_3 = C_3$ we list all walls with non-Kostant modules.

\begin{center}
\begin{tabularx}{\textwidth}{|l|X|}
\hline
\multicolumn{2}{|c|}{Type $A_3$}\\
\hline 
Block ($S$) & Non-Kostant modules ($w$)\\ 
\hline 
$\emptyset$ & $(2)$, $(31)$ \\ 
\hline 
$\{ 1 \}$ & $ (31) $ \\ 
\hline 
$\{  2 \}$ & $ (2) $ \\ 
\hline 
$\{  1,3 \}$ & $ (31) $ \\ 
\hline
\end{tabularx}
\end{center}

\begin{center}
\begin{tabularx}{\textwidth}{|l|X|}
\hline
\multicolumn{2}{|c|}{Type $B_3 = C_3$}\\
\hline 
Block ($S$) & Non-Kostant modules ($w$)\\  \hline
$\emptyset$ & $(3)$, $(2)$, $(31)$, $(23)$, $(32)$, $(232)$, $(312)$, $(231)$,
$(121)$, $(3232)$, $(1231)$, $(2312)$, $(3121)$, $(31231)$ \\ \hline
$\{1\}$ & $(31)$, $(121)$, $(231)$, $(1231)$, $(3121)$, $(31231)$ \\ \hline
$\{2\}$ & $(2)$, $(32)$, $(232)$, $(312)$, $(121)$, $(3232)$, $(2312)$,
$(3121)$ \\ \hline
$\{3\}$ & $(3)$, $(31)$, $(23)$, $(231)$, $(1231)$, $(31231)$ \\ \hline
$\{1, 2\}$ & $(121)$, $(3121)$ \\ \hline
$\{1, 3\}$ & $(31)$, $(231)$, $(1231)$, $(31231)$ \\ \hline
\end{tabularx}
\end{center}

\subsection{Rank 4}\label{s7.3}

As before, in type $A_4$, all other singular walls 
can be obtained by by exchanging $1 \leftrightarrow 4$ and $2 \leftrightarrow 3$.
In type $D_4$, all other singular walls can be obtained
permuting $1$, $3$ and $4$.

Again, 
for type $B_4$, all the walls containing non-Kostant modules are given.

\begin{center}\scriptsize
\begin{tabularx}{\textwidth}{|l|X|}
\hline
\multicolumn{2}{|c|}{Type $A_4$}\\ 
\hline 
Block ($S$) & Non-Kostant modules ($w$)\\  \hline 
$\emptyset$ & $(3)$, $(2)$, $(32)$, $(23)$, $(42)$, $(31)$, $(41)$, $(342)$,
$(232)$, $(231)$, $(423)$, $(421)$, $(312)$, $(341)$, $(431)$, $(412)$,
$(2321)$, $(2342)$, $(4232)$, $(4231)$, $(3412)$, $(1232)$, $(3431)$, $(4121)$,
$(42321)$, $(23431)$, $(34121)$, $(12342)$, $(34312)$, $(41231)$, $(343121)$,
$(123431)$ \\ \hline
$\{1\}$ & $(31)$, $(41)$, $(341)$, $(431)$, $(421)$, $(2321)$, $(3431)$,
$(4231)$, $(4121)$, $(42321)$, $(23431)$, $(41231)$, $(343121)$, $(123431)$ \\
\hline
$\{2\}$ & $(2)$, $(32)$, $(42)$, $(342)$, $(232)$, $(312)$, $(412)$,
$(2342)$, $(4232)$, $(4121)$, $(3412)$, $(1232)$, $(34121)$, $(12342)$,
$(34312)$, $(343121)$ \\ \hline
$\{1, 2\}$ & $(4121)$, $(343121)$ \\ \hline
$\{1, 3\}$ & $(31)$, $(431)$, $(2321)$, $(3431)$, $(4231)$, $(23431)$,
$(42321)$, $(41231)$, $(123431)$\\ \hline
$\{1, 4\}$ & $(41)$, $(341)$, $(421)$, $(3431)$, $(4121)$, $(343121)$,
$(123431)$ \\ \hline
$\{2, 3\}$ & $(232)$, $(4232)$, $(1232)$ \\ \hline
$\{1, 2, 4\}$ & $(4121)$, $(343121)$ \\ \hline
\end{tabularx}
\end{center}

\begin{center}\scriptsize
\begin{tabularx}{\textwidth}{|l|X|}
\hline
\multicolumn{2}{|c|}{Type $B_4 = C_4$}\\
\hline 
Block ($S$) & Non-Kostant modules ($w$)\\  \hline
$\emptyset$ & $(2)$, $(3)$, $(4)$, $(31)$, $(23)$, $(32)$, $(41)$, $(43)$, $(42)$, $(34)$, $(121)$, $(232)$, $(312)$, $(231)$, $(434)$, $(343)$, $(341)$, $(342)$, $(431)$, $(234)$, $(421)$, $(423)$, $(412)$, $(432)$, $(2312)$, $(1232)$, $(3121)$, $(4234)$, $(1231)$, $(2321)$, $(4343)$, $(4341)$, $(4342)$, $(4121)$, $(4232)$, $(4123)$, $(4312)$, $(4231)$, $(3431)$, $(2342)$, $(3423)$, $(3421)$, $(2343)$, $(2341)$, $(3432)$, $(3412)$, $(12312)$, $(23121)$, $(34234)$, $(41234)$, $(12321)$, $(43431)$, $(43423)$, $(43421)$, $(42342)$, $(43432)$, $(43412)$, $(42343)$, $(42341)$, $(12342)$, $(34123)$, $(34232)$, $(34121)$, $(12341)$, $(34312)$, $(34231)$, $(23432)$, $(23412)$, $(42312)$, $(23431)$, $(41232)$, $(43121)$, $(23423)$, $(23421)$, $(41231)$, $(42321)$, $(123121)$, $(423432)$, $(423412)$, $(423431)$, $(341234)$, $(423423)$, $(423421)$, $(434321)$, $(434123)$, $(434232)$, $(434121)$, $(412342)$, $(342342)$, $(434312)$, $(434231)$, $(412343)$, $(342343)$, $(412341)$, $(342341)$, $(123412)$, $(412312)$, $(423121)$, $(123431)$, $(123423)$, $(123421)$, $(412321)$, $(234123)$, $(234232)$, $(234121)$, $(342312)$, $(234312)$, $(234231)$, $(343121)$, $(341232)$, $(342321)$, $(341231)$, $(4123432)$, $(4123412)$, $(3423432)$, $(3423412)$, $(4123431)$, $(3423431)$, $(3423423)$, $(4123423)$, $(3423421)$, $(4123421)$, $(4234321)$, $(4234123)$, $(4234232)$, $(4234121)$, $(4343121)$, $(4341232)$, $(4234312)$, $(3412342)$, $(4234231)$, $(4342321)$, $(4341231)$, $(3412343)$, $(3412341)$, $(1234232)$, $(1234121)$, $(3412312)$, $(4342342)$, $(1234231)$, $(4342343)$, $(3412321)$, $(2342321)$, $(2341231)$, $(4123121)$, $(2342312)$, $(2341232)$, $(34123432)$, $(42342321)$, $(34123412)$, $(42341231)$, $(34123431)$, $(42342312)$, $(34123423)$, $(42343121)$, $(34123421)$, $(42341232)$, $(23412341)$, $(43412312)$, $(34234232)$, $(41234321)$, $(34234121)$, $(41234123)$, $(34234321)$, $(41234232)$, $(34234123)$, $(41234121)$, $(34234231)$, $(41234312)$, $(34234312)$, $(41234231)$, $(23412342)$, $(43412321)$, $(43423432)$, $(12342321)$, $(12341231)$, $(43423431)$, $(43423423)$, $(34123121)$, $(43412342)$, $(43412341)$, $(23412312)$, $(342342321)$, $(342341231)$, $(412342321)$, $(412341231)$, $(412343121)$, $(342343121)$, $(412341232)$, $(342341232)$, $(434123121)$, $(234123431)$, $(341234232)$, $(341234121)$, $(341234321)$, $(341234123)$, $(341234231)$, $(423412321)$, $(341234312)$, $(234123412)$, $(423412312)$, $(434234232)$, $(423412341)$, $(434234231)$, $(434234312)$, $(123412312)$, $(423412342)$, $(434123412)$, $(234123121)$, $(434123431)$, $(434123423)$, $(434123421)$, $(3412342321)$, $(3412341231)$, $(4234123121)$, $(3412343121)$, $(3412341232)$, $(2341234121)$, $(3423412321)$, $(4123412321)$, $(2341234312)$, $(4123412312)$, $(4342342321)$, $(1234123121)$, $(4342341232)$, $(3423412342)$, $(4341234232)$, $(4341234121)$, $(4234123431)$, $(4341234231)$, $(4234123412)$, $(41234123121)$, $(43423412342)$, $(23412343121)$, $(34123412321)$, $(42341234121)$, $(42341234312)$, $(43412342321)$, $(43412341231)$, $(34234123431)$, $(434234123431)$, $(423412343121)$ \\ \hline

$\{1\}$ & $(31)$, $(41)$, $(231)$, $(121)$, $(341)$, $(431)$,
$(421)$, $(3121)$, $(1231)$, $(2321)$, $(4341)$, $(3431)$, $(3421)$, $(2341)$,
$(4231)$, $(4121)$, $(23121)$, $(12321)$, $(42341)$, $(43431)$, $(34121)$,
$(12341)$, $(34231)$, $(23431)$, $(43121)$, $(23421)$, $(41231)$, $(42321)$,
$(123121)$, $(423431)$, $(423421)$, $(434321)$, $(434121)$, $(412341)$,
$(342341)$, $(123421)$, $(341231)$, $(342321)$, $(123431)$, $(343121)$,
$(234121)$, $(423121)$, $(412321)$, $(234231)$, $(4234321)$, $(4234121)$,
$(3412341)$, $(4234231)$, $(4341231)$, $(4342321)$, $(4123421)$, $(3423421)$,
$(4343121)$, $(4123431)$, $(3423431)$, $(1234121)$, $(1234231)$, $(3412321)$,
$(2342321)$, $(2341231)$, $(4123121)$, $(42342321)$, $(42341231)$, $(34123431)$,
$(42343121)$, $(34123421)$, $(23412341)$, $(41234321)$, $(34234321)$,
$(41234121)$, $(41234231)$, $(43412321)$, $(12341231)$, $(12342321)$,
$(34123121)$, $(43423431)$, $(43412341)$, $(341234321)$, $(341234121)$,
$(341234231)$, $(423412321)$, $(234123431)$, $(434123121)$, $(342342321)$,
$(412341231)$, $(412342321)$, $(412343121)$, $(342343121)$, $(423412341)$,
$(434123431)$, $(434123421)$, $(3412342321)$, $(3412341231)$, $(3412343121)$,
$(3423412321)$, $(4123412321)$, $(4342342321)$, $(4234123431)$, $(1234123121)$,
$(4341234121)$, $(4341234231)$, $(34123412321)$, $(23412343121)$,
$(41234123121)$, $(43412342321)$, $(43412341231)$, $(34234123431)$,
$(434234123431)$, $(423412343121)$ \\ \hline

$\{2\}$ & $(2)$, $(32)$, $(42)$, $(312)$, $(121)$, $(232)$,
$(342)$, $(412)$, $(432)$, $(2312)$, $(1232)$, $(3121)$, $(4342)$, $(3432)$,
$(3412)$, $(4312)$, $(2342)$, $(4121)$, $(4232)$, $(42342)$, $(12312)$,
$(23121)$, $(43432)$, $(43412)$, $(12342)$, $(34232)$, $(34121)$, $(34312)$,
$(23432)$, $(23412)$, $(42312)$, $(41232)$, $(43121)$, $(123121)$, $(423432)$,
$(423412)$, $(434232)$, $(434121)$, $(412342)$, $(342342)$, $(434312)$,
$(343121)$, $(341232)$, $(123412)$, $(342312)$, $(234312)$, $(234232)$,
$(234121)$, $(412312)$, $(423121)$, $(4234312)$, $(3412342)$, $(4234232)$,
$(4234121)$, $(4343121)$, $(4341232)$, $(4123432)$, $(3423432)$, $(4123412)$,
$(3423412)$, $(1234232)$, $(1234121)$, $(3412312)$, $(4342342)$, $(4123121)$,
$(2342312)$, $(2341232)$, $(34123432)$, $(34123412)$, $(42342312)$,
$(42343121)$, $(42341232)$, $(43412312)$, $(34234232)$, $(34234121)$,
$(41234232)$, $(41234121)$, $(41234312)$, $(34234312)$, $(23412342)$,
$(43423432)$, $(34123121)$, $(23412312)$, $(43412342)$, $(423412312)$,
$(341234312)$, $(341234232)$, $(341234121)$, $(234123412)$, $(412343121)$,
$(412341232)$, $(342343121)$, $(342341232)$, $(434123121)$, $(434234232)$,
$(434234312)$, $(123412312)$, $(423412342)$, $(434123412)$, $(234123121)$,
$(4234123121)$, $(3412343121)$, $(3412341232)$, $(2341234121)$, $(2341234312)$,
$(4123412312)$, $(4342341232)$, $(4234123412)$, $(1234123121)$, $(3423412342)$,
$(4341234232)$, $(4341234121)$, $(43423412342)$, $(23412343121)$,
$(41234123121)$, $(42341234121)$, $(42341234312)$, $(423412343121)$  \\ \hline

$\{3\}$ & $(3)$, $(31)$, $(23)$, $(43)$, $(231)$, $(232)$,
$(343)$, $(431)$, $(423)$, $(1232)$, $(1231)$, $(2321)$, $(4343)$, $(3431)$,
$(3423)$, $(2343)$, $(4231)$, $(4123)$, $(4232)$, $(12312)$, $(12321)$,
$(42343)$, $(43431)$, $(43423)$, $(34123)$, $(34232)$, $(34231)$, $(23431)$,
$(41232)$, $(23423)$, $(41231)$, $(42321)$, $(123121)$, $(423431)$, $(423423)$,
$(434123)$, $(434232)$, $(434231)$, $(412343)$, $(342343)$, $(123423)$,
$(342321)$, $(341231)$, $(123431)$, $(341232)$, $(234232)$, $(234123)$,
$(412312)$, $(412321)$, $(234231)$, $(4234232)$, $(4234123)$, $(3412343)$,
$(4234231)$, $(4342321)$, $(4341231)$, $(4123423)$, $(3423423)$, $(4341232)$,
$(4123431)$, $(3423431)$, $(1234232)$, $(3412312)$, $(1234231)$, $(4342343)$,
$(3412321)$, $(2342321)$, $(2341231)$, $(4123121)$, $(42342321)$, $(42341231)$,
$(34123431)$, $(34123423)$, $(42341232)$, $(43412312)$, $(34234232)$,
$(41234123)$, $(41234232)$, $(34234123)$, $(34234231)$, $(41234231)$,
$(43412321)$, $(12342321)$, $(12341231)$, $(34123121)$, $(43423423)$,
$(43423431)$, $(23412312)$, $(341234232)$, $(341234123)$, $(423412312)$,
$(423412321)$, $(341234231)$, $(434123121)$, $(234123431)$, $(412342321)$,
$(342341231)$, $(342342321)$, $(412341231)$, $(412341232)$, $(342341232)$,
$(434234232)$, $(434234231)$, $(123412312)$, $(234123121)$, $(434123431)$,
$(434123423)$, $(3412342321)$, $(3412341231)$, $(4234123121)$, $(3412341232)$,
$(3423412321)$, $(4123412321)$, $(4123412312)$, $(4342342321)$, $(4234123431)$,
$(4342341232)$, $(1234123121)$, $(4341234232)$, $(4341234231)$, $(34123412321)$,
$(41234123121)$, $(43412342321)$, $(43412341231)$, $(34234123431)$,
$(434234123431)$ \\ \hline

$\{4\}$ & $(4)$, $(41)$, $(34)$, $(42)$, $(341)$, $(421)$,
$(234)$, $(412)$, $(342)$, $(434)$, $(4234)$, $(4343)$, $(4341)$, $(4342)$,
$(4121)$, $(2342)$, $(3421)$, $(3412)$, $(2341)$, $(23412)$, $(12342)$,
$(34121)$, $(12341)$, $(23421)$, $(34234)$, $(41234)$, $(43431)$, $(43421)$,
$(42342)$, $(43432)$, $(43412)$, $(42343)$, $(42341)$, $(423432)$, $(423412)$,
$(423431)$, $(341234)$, $(423421)$, $(434321)$, $(434121)$, $(412342)$,
$(342342)$, $(434312)$, $(412343)$, $(342343)$, $(412341)$, $(342341)$,
$(123412)$, $(234121)$, $(123421)$, $(1234121)$, $(4342343)$, $(4342342)$,
$(4123432)$, $(4123412)$, $(3423432)$, $(3423412)$, $(4123431)$, $(3423431)$,
$(3423421)$, $(4123421)$, $(4234321)$, $(4234121)$, $(4343121)$, $(4234312)$,
$(3412342)$, $(3412343)$, $(3412341)$, $(34123432)$, $(34123412)$, $(34123431)$,
$(42343121)$, $(34123421)$, $(23412341)$, $(41234321)$, $(34234121)$,
$(34234321)$, $(41234121)$, $(41234312)$, $(34234312)$, $(43423432)$,
$(43412342)$, $(43423431)$, $(43412341)$, $(423412341)$, $(434234312)$,
$(423412342)$, $(434123412)$, $(434234232)$, $(434123421)$, $(412343121)$,
$(342343121)$, $(234123431)$, $(341234121)$, $(341234321)$, $(341234312)$,
$(234123412)$, $(3412343121)$, $(2341234121)$, $(2341234312)$, $(4234123431)$,
$(3423412342)$, $(4341234121)$, $(4234123412)$, $(42341234121)$,
$(42341234312)$, $(43423412342)$, $(23412343121)$, $(434234123431)$,
$(423412343121)$ \\ \hline

$\{1, 2\}$ & $(121)$, $(3121)$, $(4121)$, $(23121)$,
$(34121)$, $(43121)$, $(123121)$, $(434121)$, $(343121)$, $(234121)$,
$(423121)$, $(4234121)$, $(4343121)$, $(1234121)$, $(4123121)$, $(42343121)$,
$(41234121)$, $(434123121)$, $(412343121)$, $(342343121)$, $(3412343121)$,
$(1234123121)$, $(4341234121)$, $(23412343121)$, $(41234123121)$,
$(423412343121)$ \\ \hline

$\{1, 3\}$ & $(31)$, $(231)$, $(431)$, $(1231)$, $(2321)$,
$(3431)$, $(4231)$, $(12321)$, $(43431)$, $(34231)$, $(23431)$, $(41231)$,
$(42321)$, $(123121)$, $(423431)$, $(342321)$, $(341231)$, $(123431)$,
$(412321)$, $(234231)$, $(4234231)$, $(4342321)$, $(4341231)$, $(4123431)$,
$(3423431)$, $(1234231)$, $(3412321)$, $(2342321)$, $(2341231)$, $(4123121)$,
$(42342321)$, $(42341231)$, $(34123431)$, $(41234231)$, $(43412321)$,
$(12342321)$, $(12341231)$, $(34123121)$, $(43423431)$, $(341234231)$,
$(423412321)$, $(234123431)$, $(434123121)$, $(342342321)$, $(412342321)$,
$(412341231)$, $(434123431)$, $(3412342321)$, $(3412341231)$, $(4123412321)$,
$(3423412321)$, $(4342342321)$, $(4234123431)$, $(1234123121)$, $(4341234231)$,
$(34123412321)$, $(41234123121)$, $(43412342321)$, $(43412341231)$,
$(34234123431)$, $(434234123431)$ \\ \hline

$\{1, 4\}$ & $(41)$, $(341)$, $(421)$, $(2341)$, $(4121)$,
$(3421)$, $(4341)$, $(34121)$, $(42341)$, $(12341)$, $(23421)$, $(43431)$,
$(234121)$, $(123421)$, $(423431)$, $(423421)$, $(434321)$, $(434121)$,
$(412341)$, $(342341)$, $(1234121)$, $(4234321)$, $(4234121)$, $(3412341)$,
$(4123421)$, $(3423421)$, $(4343121)$, $(4123431)$, $(3423431)$, $(34123431)$,
$(42343121)$, $(43423431)$, $(34123421)$, $(23412341)$, $(41234321)$,
$(34234321)$, $(41234121)$, $(43412341)$, $(341234321)$, $(341234121)$,
$(423412341)$, $(234123431)$, $(412343121)$, $(342343121)$, $(434123421)$,
$(3412343121)$, $(4341234121)$, $(4234123431)$, $(23412343121)$,
$(434234123431)$, $(423412343121)$ \\ \hline

$\{2, 3\}$ & $(232)$, $(1232)$, $(4232)$, $(12312)$,
$(34232)$, $(41232)$, $(123121)$, $(434232)$, $(341232)$, $(234232)$,
$(412312)$, $(4234232)$, $(4341232)$, $(1234232)$, $(3412312)$, $(4123121)$,
$(42341232)$, $(43412312)$, $(34234232)$, $(41234232)$, $(34123121)$,
$(23412312)$, $(341234232)$, $(423412312)$, $(434123121)$, $(412341232)$,
$(342341232)$, $(434234232)$, $(123412312)$, $(234123121)$, $(4234123121)$,
$(3412341232)$, $(4123412312)$, $(1234123121)$, $(4342341232)$, $(4341234232)$,
$(41234123121)$ \\ \hline

$\{2, 4\}$ & $(42)$, $(412)$, $(342)$, $(3412)$, $(4121)$,
$(2342)$, $(23412)$, $(12342)$, $(34121)$, $(42342)$, $(43432)$, $(43412)$,
$(423432)$, $(423412)$, $(123412)$, $(234121)$, $(434121)$, $(412342)$,
$(342342)$, $(434312)$, $(1234121)$, $(4234312)$, $(3412342)$, $(4234121)$,
$(4343121)$, $(4342342)$, $(4123432)$, $(3423432)$, $(4123412)$, $(3423412)$,
$(34123432)$, $(34123412)$, $(42343121)$, $(43423432)$, $(34234121)$,
$(41234121)$, $(43412342)$, $(41234312)$, $(34234312)$, $(341234312)$,
$(434234312)$, $(341234121)$, $(423412342)$, $(234123412)$, $(434123412)$,
$(434234232)$, $(412343121)$, $(342343121)$, $(4234123412)$, $(3412343121)$,
$(3423412342)$, $(2341234121)$, $(4341234121)$, $(2341234312)$, $(43423412342)$,
$(42341234121)$, $(23412343121)$, $(42341234312)$, $(423412343121)$ \\ \hline

$\{3, 4\}$ & $(4343)$, $(43431)$, $(42343)$, $(423431)$,
$(412343)$, $(342343)$, $(3412343)$, $(4123431)$, $(3423431)$, $(4342343)$,
$(34123431)$, $(43423431)$, $(234123431)$, $(4234123431)$, $(434234123431)$ \\ \hline
$\{1, 2, 3\}$ & $(123121)$, $(4123121)$, $(434123121)$,
$(1234123121)$, $(41234123121)$ \\ \hline	

$\{1, 2, 4\}$ & $(4121)$, $(34121)$, $(234121)$, $(434121)$,
$(1234121)$, $(4234121)$, $(4343121)$, $(42343121)$, $(41234121)$,
$(412343121)$, $(342343121)$, $(3412343121)$, $(4341234121)$, $(23412343121)$,
$(423412343121)$ \\ \hline	

$\{1, 3, 4\}$ & $(43431)$, $(423431)$, $(4123431)$,
$(3423431)$, $(34123431)$, $(43423431)$, $(234123431)$, $(4234123431)$,
$(434234123431)$ \\ \hline
\end{tabularx}
\end{center}

\begin{center}\scriptsize  
\begin{tabularx}{\textwidth}{|l|X|}
\hline
\multicolumn{2}{|c|}{Type $D_4$}\\
\hline 
Block ($S$) & Non-Kostant modules ($w$)\\  \hline 
$\emptyset$ & $(1)$, $(2)$, $(3)$, $(4)$, $(41)$, $(21)$, $(24)$, $(23)$, $(12)$,
$(42)$, $(31)$, $(43)$, $(32)$, $(121)$, $(242)$, $(412)$, $(241)$, $(431)$,
$(232)$, $(312)$, $(432)$, $(231)$, $(243)$, $(2412)$, $(2312)$, $(2432)$,
$(4121)$, $(1242)$, $(4232)$, $(2431)$, $(2421)$, $(1241)$, $(2321)$, $(3243)$,
$(3121)$, $(3242)$, $(1232)$, $(2423)$, $(4312)$, $(1231)$, $(24121)$,
$(12412)$, $(12312)$, $(41232)$, $(12431)$, $(12421)$, $(32423)$, $(12321)$,
$(24232)$, $(23121)$, $(32432)$, $(43121)$, $(31242)$, $(24312)$, $(24231)$,
$(32431)$, $(124121)$, $(324232)$, $(324312)$, $(324231)$, $(242312)$,
$(312423)$, $(412321)$, $(243121)$, $(231242)$, $(123121)$, $(124312)$,
$(124231)$, $(312431)$, $(312421)$, $(241232)$, $(3124232)$, $(4123121)$,
$(1241231)$, $(2312431)$, $(3124231)$, $(4231242)$, $(1243121)$, $(3241232)$,
$(3124121)$, $(3242321)$, $(42312431)$, $(31242321)$, $(31241231)$ \\ \hline
$\{1\}$ & $(1)$, $(41)$, $(21)$, $(31)$, $(241)$, $(231)$,
$(121)$, $(431)$, $(2431)$, $(2421)$, $(1241)$, $(2321)$, $(1231)$, $(12421)$,
$(32431)$, $(12321)$, $(24231)$, $(43121)$, $(12431)$, $(124121)$, $(412321)$,
$(123121)$, $(124231)$, $(312431)$, $(312421)$, $(3242321)$, $(1241231)$,
$(3124231)$, $(3124121)$, $(4123121)$, $(2312431)$, $(1243121)$, $(42312431)$,
$(31242321)$, $(31241231)$ \\ \hline

$\{2\}$ & $(2)$, $(12)$, $(42)$, $(32)$, $(412)$, $(121)$,
$(242)$, $(232)$, $(312)$, $(432)$, $(2412)$, $(2312)$, $(2432)$, $(4121)$,
$(1242)$, $(4232)$, $(3121)$, $(3242)$, $(1232)$, $(4312)$, $(24312)$,
$(24121)$, $(12412)$, $(23121)$, $(32432)$, $(43121)$, $(31242)$, $(41232)$,
$(24232)$, $(12312)$, $(124121)$, $(324232)$, $(324312)$, $(242312)$,
$(243121)$, $(231242)$, $(123121)$, $(124312)$, $(241232)$, $(4231242)$,
$(3124232)$, $(4123121)$, $(1243121)$, $(3241232)$, $(3124121)$ \\ \hline

$\{1, 2\}$ & $(121)$, $(43121)$, $(124121)$, $(123121)$,
$(4123121)$, $(3124121)$, $(1243121)$ \\ \hline

$\{1, 3\}$ & $(31)$, $(431)$, $(231)$, $(2431)$, $(1231)$,
$(2321)$, $(24231)$, $(12431)$, $(32431)$, $(12321)$, $(123121)$, $(412321)$,
$(312431)$, $(1241231)$, $(3242321)$, $(4123121)$, $(2312431)$, $(3124231)$,
$(42312431)$, $(31241231)$, $(31242321)$ \\ \hline

$\{1, 2, 3\}$ & $(123121)$, $(4123121)$ \\ \hline

$\{1, 3, 4\}$ & $(431)$, $(12431)$, $(24231)$, $(32431)$,
$(1241231)$, $(3242321)$, $(2312431)$, $(3124231)$, $(42312431)$, $(31241231)$,
$(31242321)$ \\ \hline

\end{tabularx}
\end{center}

Note that most of the BGG complexes resolving a module of 
a dominant singular highest weight are not exact, in contrast 
to the regular case where all the BGG complexes of modules 
of a dominant highest weight are, in fact, exact.

For the exceptional case $F_4$, due to the computational 
complexity, we present only one type of singularity, which 
has the least number of maximal representatives. However, 
the full exposition of all the blocks would probably take 
up several pages.

\begin{center}\scriptsize  
\begin{tabularx}{\textwidth}{|l|X|}
\hline
\multicolumn{2}{|c|}{Type $F_4$}\\
\hline 
Block ($S$) & Non-Kostant modules ($w$)\\  \hline 
$\{2, 3, 4\}$ & $(432343232)$, $(4312343232)$,
$(42312343232)$, $(342312343232)$, $(2342312343232)$, $(3432312343232)$,
$(12342312343232)$, $(23432312343232)$, $(123432312343232)$,
$(3123432312343232)$, $(3234323123431232)$, $(31234323123431232)$,
$(231234323123431232)$ \\ \hline
\end{tabularx}
\end{center}

\section{BGG complexes for balanced quasi-hereditary algebras}\label{s8}

\subsection{Balanced quasi-hereditary algebras}\label{s8.1}

In this section we work in the setup of \cite{mazorchuk2009koszul}. 
Let $A$ be a finite dimensional positively graded quasi-hereditary 
algebra over an algebraically closed field. Denote by 
$\{e_\lambda \colon \lambda \in \Lambda\}$ a complete set 
of pairwise-orthogonal primitive idempotents of $A$ with 
a fixed linear order on $\Lambda$ that
defines the quasi-hereditary structure. Denote by 
$A$-${\operatorname{gmod}}$ the category of all 
finite-dimensional graded $A$-modules, where morphisms are 
homogeneous homomorphisms of degree zero. This is an abelian 
category with enough projectives and enough injectives. 
Denote by $\shift{\cdot}$ the shift in grading: 
$M\shift{i}_j := M_{i+j}$. Denote by $L(\lambda)$, 
$\Delta(\lambda)$, $\nabla(\lambda)$, $P(\lambda)$, 
$I(\lambda)$, and $T(\lambda)$ respectively the simple, 
standard, costandard, projective, injective, and 
tilting module corresponding to $\lambda \in \Lambda$. 
We fix their graded shifts so that $L(\lambda)$, 
$\Delta(\lambda)$ and  $P(\lambda)$ have top in degree zero,
$\nabla(\lambda)$ and $I(\lambda)$ have socle in degree zero,
and $T(\lambda)$ has in degree zero the unique subquotient
isomorphic to $L(\lambda)$.
These modules are called the \emph{structural modules}. If $M$ 
is a structural module, we will say that $M\shift{i}$ is 
{\em centered} at $-i$.

Denote by $\mathcal{D}^b(A)$ the bounded derived category of 
$A$-${\operatorname{gmod}}$. A complex
\begin{displaymath}
X^\bullet = \quad \ldots \to X^{i+1} \to X^{i} \to X^{i-1} \to \ldots
\end{displaymath}
of direct sums of structural modules of the same kind is 
said to be \emph{linear}, provided that all indecomposable 
direct summands of each $X^i$ are centered at $i$. 
Denote by $[i]$ the homological shift normalized such that 
$X[i]^j = X^{i+j}$.

Assume that $A$ is \emph{balanced} in the sense of 
\cite{mazorchuk2009koszul}. This means that each standard 
module $\Delta(\lambda)$ has a linear tilting coresolution 
$0 \to \Delta(\lambda) \to {\mathcal{S}}_\lambda^\bullet$, and
each costandard module $\nabla(\lambda)$ has a linear 
tilting resolution $\mathcal{C}_\lambda^\bullet \to \Delta(\lambda) \to 0$. 
The main examples for our purposes are graded versions of blocks of 
the category ${\mathcal{O}}$, see Subsection~\ref{subsection:graded_O}. 
That these are given by balanced algebras is proved in 
\cite[Proposition~2.7]{mazorchuk2009applications} for the 
regular block, and \cite[Section~4]{mazorchuk2009applications} 
for singular blocks.

In what follows we will use  the following statement, cf. 
\cite[Corollary~7 and Proposition 5]{mazorchuk2009koszul}:

\begin{proposition} \label{propnew76142}
For a balanced quasi-hereditary algebra $A$, the following holds:
\begin{enumerate}[$($a$)$]
\item\label{propnew76142.1} $A$ is \emph{Koszul}, i.e. 
the minimal projective resolution of each $L(\lambda)$ is linear.
\item\label{propnew76142.2} Each $L(\lambda)$ is isomorphic in 
$\mathcal{D}^b(A)$ to a linear complex of tilting modules. 
In other words, there is a linear complex 
$\mathcal{T}_\lambda^\bullet$ of tilting modules such that
\begin{displaymath}
H_i(\mathcal{T}_\lambda^\bullet) = 
\begin{cases} 
L(\lambda), & \text{ if } i=0; \\
0, &  \text{ if } i \neq 0. 
\end{cases} 
\end{displaymath}
\end{enumerate}
\end{proposition}

Let us mention that a balanced algebra is also necessarily 
\emph{standard Koszul} in the sense of \cite{ADL}, which 
means that standard modules have linear projective resolutions, 
and costandard modules have linear injective coresolutions. 
This is stronger than just Koszulity.

\subsection{BGG complexes for balanced algebras}\label{s8.2}

In this setup, by a BGG complex of $L(\lambda)$ we mean any 
linear complex
\begin{displaymath}
\Delta_\lambda^\bullet = \quad \ldots \to \Delta^1_\lambda 
\to \Delta^0_\lambda \to L(\lambda) \to 0, 
\end{displaymath}
where each $\Delta^i_\lambda$ is a direct sum of standard modules, 
and $\Delta^0_\lambda \to L(\lambda)$ is an epimorphism. If such a
complex happens to be exact, we say it is a BGG resolution 
of $L(\lambda)$.

\begin{proposition}\label{item:prop_gh}
Let $\Delta_\lambda^\bullet$ be a BGG complex of $L(\lambda)$.
\begin{enumerate}[$($a$)$]
\item \label{item:prop_gh1}
There exists a unique (up to scalar) non-trivial homomorphism 
$\xi:{\mathcal{P}}_\lambda^\bullet \to \Delta_\lambda^\bullet$ in the
category of complexes.
\item \label{item:prop_gh2}
There exists a unique (up to scalar) non-trivial homomorphism 
$\eta:{\mathcal{P}}_\lambda^\bullet \to \mathcal{T}_\lambda^\bullet$ in
the homotopy category. This homomorphism is a quasi-isomorphism.
\item \label{item:prop_gh3}
For each homomorphism 
$\tilde{\eta}:{\mathcal{P}}_\lambda^\bullet \to \mathcal{T}_\lambda^\bullet$
in the category of complexes descending to $\eta$ in the homotopy category,
there exist a homomorphism  
$\varphi:\Delta_\lambda^\bullet {\longrightarrow} 
\mathcal{T}_\lambda^\bullet$ of complexes such that the diagram
\begin{equation}\label{eq:prop_gh3}
\xymatrix{ {\mathcal{P}}_\lambda^\bullet \ar[r]^{\xi} \ar[rd]_{\tilde{\eta}} & 
\Delta_\lambda^\bullet \ar[d]^{\varphi} \\ & \mathcal{T}_\lambda^\bullet}
\end{equation}
commutes in the category of complexes. If, moreover, 
$\Delta_\lambda^\bullet$  is exact, then such $\varphi$ is unique and injective, and $\xi$ is surjective.
\end{enumerate}
\end{proposition}

\begin{proof}
Existence parts of claims \eqref{item:prop_gh1} and \eqref{item:prop_gh2} 
follow from  the fact that ${\mathcal{P}}_\lambda^\bullet$ consists of 
projective modules. The uniqueness part in \eqref{item:prop_gh1}
follows from the linearity of involved complexes. Indeed, since the 
modules in both complexes are centered at the position of their heads, 
from the positivity of the grading it follows that there are no 
non-trivial homotopies between ${\mathcal{P}}_\lambda^\bullet$ and 
$\Delta_\lambda^\bullet$. The uniqueness part in \eqref{item:prop_gh2}
follows from Schur's Lemma.

To prove claim~\eqref{item:prop_gh3}, note that there exists a map 
$\Delta^0_\lambda \to L(\lambda)$, which gives a morphism
in  ${\operatorname{Hom}}_{\mathcal{D}^b(A)}
(\Delta_\lambda^\bullet,\mathcal{T}_\lambda^\bullet)$. 
Since tilting modules have costandard filtrations
and since standard objects are left orthogonal 
to costandard objects, the existence of $\varphi$
follows from \cite[Lemma~III.2.1]{happel1988triangulated}.

If $\Delta_\lambda^\bullet$ is exact, then
all maps in \eqref{eq:prop_gh3} are isomorphisms in 
$\mathcal{D}^b(A)$, in particular, 
for each $\mu \in \Lambda$, $i\in\mathbb{Z}_{\geq 0}$
and $j\in\mathbb{Z}$, the map
\begin{equation}\label{eqnew5523}
\xymatrix{{\operatorname{Hom}}_{\mathcal{D}^b(A)}
(\mathcal{T}_\lambda^\bullet, \nabla(\mu)[i]\langle j\rangle) 
\ar[r]^{- \circ \varphi} &
{\operatorname{Hom}}_{\mathcal{D}^b(A)}
(\Delta_\lambda^\bullet, \nabla(\mu)[i]\langle j\rangle) } 
\end{equation}
is an isomorphism. 

Let $\Delta(\mu)\langle -i\rangle$ be a direct summand of $\Delta_{\lambda}^i$.
Then the unique up to scalar non-zero map $f:\Delta(\mu)\langle -i\rangle\to
\nabla(\mu)\langle -i\rangle$ gives rise to a non-zero element in 
\begin{equation}\label{eqeqnder}
\mathrm{Hom}_{\mathcal{D}^b(A)}(L(\lambda),\nabla(\mu)\langle -i\rangle [i]). 
\end{equation}
This is
due to the combination of \cite[Lemma~III.2.1]{happel1988triangulated},
which allows us to compute extensions directly in the homotopy category
of complexes and the fact that the linearity of $\Delta_\lambda^\bullet$ 
does not allow $f$ to be killed by any homotopy. Note also that 
the extension \eqref{eqeqnder} can also
be computed in the homotopy category using $\mathcal{T}_\lambda^\bullet$,
again thanks to  \cite[Lemma~III.2.1]{happel1988triangulated}.
And again, the positivity of the grading does not allow for any 
homotopies from $\mathcal{T}_\lambda^\bullet$ to $\nabla(\mu)\langle -i\rangle$.
Consequently, by \eqref{eqnew5523}, the restriction of $\varphi$ to the summand 
$\Delta(\mu)\langle -i\rangle$ is uniquely defined. 
From this it follows that $\varphi$ is unique and injective. Similarly one sees that $\xi$ is surjective.
\end{proof}

\subsection{Uniqueness of BGG resolution}\label{s8.3}

Proposition~\ref{item:prop_gh} has the following consequence.

\begin{theorem}\label{thmuniquebgg}
Assume that  $\Delta_\lambda^\bullet$ is exact. Then all 
BGG resolutions of $L(\lambda)$ are isomorphic in the category of complexes.
\end{theorem}

\begin{proof}
By  Proposition~\ref{item:prop_gh}, any BGG resolution of $L(\lambda)$
is a subcomplex of $\mathcal{T}_\lambda^\bullet$, namely, the image of 
the injective morphism $\varphi$  constructed in the proof. Since
the Ringel dual of a balanced quasi-hereditary algebra is
positively graded, see \cite[Theorem~1(ii)]{mazorchuk2009koszul},
and both $\Delta_\lambda^\bullet$ and $\mathcal{T}_\lambda^\bullet$
are linear, there are no homotopies from 
$\Delta_\lambda^\bullet$ to $\mathcal{T}_\lambda^\bullet$.
Therefore the morphism $\varphi$ in the proof of
Proposition~\ref{item:prop_gh} does not depend on the choice of 
$\tilde{\eta}$. In particular, the image of the injective morphism
$\varphi$ is a canonical subcomplex of $\mathcal{T}_\lambda^\bullet$
isomorphic to $\Delta_\lambda^\bullet$ and independent of
the choice of $\Delta_\lambda^\bullet$.
\end{proof}

\section{BGG complexes for ${\mathcal{S}}$-subcategories in ${\mathcal{O}}$}\label{s9}

\subsection{${\mathcal{S}}$-subcategories in ${\mathcal{O}}$}\label{s9.1}

In this section we will apply our results on singular BGG resolutions to
construct analogues of BGG resolutions for certain generalization of 
category $\mathcal{O}$ which are no longer described by quasi-hereditary
algebras but rather by the so-called {\em standardly stratified} algebras,
see \cite{FKM1}, or, even, {\em properly stratified} algebras, see \cite{Dl}. 
These are the so-called 
{\em ${\mathcal{S}}$-subcategories in ${\mathcal{O}}$} as defined in 
\cite{futorny2000s}. The interest in ${\mathcal{S}}$-subcategories in 
${\mathcal{O}}$ is motivated by the study of the structure of parabolically
induced modules for Lie algebra, see \cite{MS} and references therein. 
As it turns out, 
in many cases, parabolically induced modules naturally belong to
certain categories equivalent to ${\mathcal{S}}$-subcategories in 
${\mathcal{O}}$. This, in particular, implies that composition multiplicities
of parabolically induced modules can be described using KL-polynomials.
Let us briefly recall the definition of ${\mathcal{S}}$-subcategories in 
${\mathcal{O}}$.

Let $\lambda$ be a dominant integral weight and $W_{\lambda}$ the stabilizer
of $\lambda+\rho$. Denote by ${}^{\lambda}\tilde{W}$ the set of longest
coset representatives in $W_{\lambda}\backslash W$. Let $\mathcal{I}_{\lambda}$
denote the Serre subcategory of $\mathcal{O}_{0}$ generated by all
$L(w\cdot 0)$, where $w\notin {}^{\lambda}\tilde{W}$. Then the Serre quotient category
$\mathcal{S}_{\lambda}:=\mathcal{O}_{0}/\mathcal{I}_{\lambda}$ is an abelian
category and we denote by $\pi:\mathcal{O}_{0}\to \mathcal{S}_{\lambda}$
the canonical projection. 

Let $A$ be a basic finite dimensional associative algebra such that
$A$-mod is equivalent to $\mathcal{O}_{0}$. Let $\displaystyle 1=\sum_{w\in W}e_w$ 
be a primitive decomposition of the identity in $A$, where $e_w$ corresponds to
$L(w\cdot 0)$. Then $\mathcal{S}_{\lambda}$ is equivalent to $eAe$-mod,
where 
\begin{displaymath}
e=e(\lambda)=\sum_{w\in{}^{\lambda}\tilde{W}}e_w.
\end{displaymath}
The algebra $eAe$ is properly stratified in the sense of \cite{Dl}
with $\pi(\Delta(w\cdot 0))$, where $w\in {}^{\lambda}\tilde{W}$, 
being a complete and irredundant list of representatives of proper 
standard objects with respect to the properly stratified structure. 
Note that, for $w_1,w_2\in W$, we have  
$\pi(\Delta(w_1\cdot 0))\cong \pi(\Delta(w_2\cdot 0))$ if
and only if $w_1=xw_2$, for some $x\in W_{\lambda}$.

From the above, it follows that the set 
$\{\pi(L(w\cdot 0))\,:\,w\in {}^{\lambda}\tilde{W}\}$ 
is a complete and irredundant 
list of representatives of isomorphism classes of simple objects in 
$\mathcal{S}_{\lambda}$.

\subsection{Soergel's equivalence}\label{s9.2}

Denote by $\underline{\mathcal{O}}$ the full subcategory in $\mathcal{O}$
consisting of all modules with the property that the center of the
universal enveloping algebra acts by scalars on each indecomposable
direct summand. Note that all simple and all Verma modules are in
$\underline{\mathcal{O}}$. For simplicity, let us restrict to 
the principal block of $\mathcal{O}$ and $\underline{\mathcal{O}}$. 

Recall, see \cite[Kapitel~6]{Ja}, that $\mathcal{O}_0$ is equivalent to the
category ${}^{\infty}_{\,\,\,0}\mathcal{H}_{0}^1$ of finitely generated 
Harish-Chandra  $\mathfrak{g}$-$\mathfrak{g}$-bimodules having generalized
central character of $\mathcal{O}_0$ on the left and genuine 
central character of $\mathcal{O}_0$ on the right. Under this equivalence,
the subcategory $\underline{\mathcal{O}}_0$ corresponds to 
the full subcategory ${}^{1}_0\mathcal{H}_{0}^1$ of 
${}^{\infty}_{\,\,\,0}\mathcal{H}_{0}^1$ consisting of all objects having a 
genuine central character on the left. As pointed out in \cite{So0}, 
swapping the sides of bimodules induces a self-equivalence of 
${}^{1}_0\mathcal{H}_{0}^1$ and we obtain the following claim:

\begin{proposition}[\cite{So0}]\label{propsoer}
There is an involutive self-equivalence $\Psi$ of $\underline{\mathcal{O}}_0$
such that $\Psi(L(w\cdot 0))\cong L(w^{-1}\cdot 0)$.
\end{proposition}

Note that $\Psi$ does not extend to ${}^{\infty}_{\,\,\,0}\mathcal{H}_{0}^1$
or ${\mathcal{O}}_0$ due to asymmetry of the requirements for the left and right 
actions of the center in the definition of ${}^{\infty}_{\,\,\,0}\mathcal{H}_{0}^1$.

\subsection{BGG complexes and resolutions}\label{s9.3}

A {\em BGG complex} in $\mathcal{S}_{\lambda}$ is a complex in which 
each component is isomorphic to a direct sum of proper standard objects.
If a {\em BGG complex} has a unique homology and that homology is isomorphic to
$\pi(L(w\cdot 0))$, where $w\in {}^{\lambda}\tilde{W}$, we will call such a
complex a {\em BGG resolution} of $\pi(L(w\cdot 0))$. Now we can formulate
our main result in this section.

\begin{theorem}\label{thmnew9557}
For $w\in {}^{\lambda}\tilde{W}$, the following assertions are equivalent:
\begin{enumerate}[$($a$)$]
\item\label{thmnew9557.1} $\pi(L(w\cdot 0))$ has a BGG resolution.
\item\label{thmnew9557.2} $L(w^{-1}\cdot \lambda)$ is a Kostant module
in the sense of Subsection~\ref{subsection:Kostant}.
\end{enumerate}
\end{theorem}

\begin{proof}
Assume claim~\eqref{thmnew9557.1}. Let 
\begin{displaymath}
Q^{\bullet}:
\dots\to Q^2\to Q^1\to Q^1\to \pi(L(w\cdot 0))\to 0 
\end{displaymath}
be a BGG resolution of $\pi(L(w\cdot 0))$ in $\mathcal{S}_{\lambda}$.
Since all tops and all socles for all components in $Q^{\bullet}$,
considered as objects in $\mathcal{O}$, are simple modules of the form 
$L(x\cdot 0)$, with $x\in {}^{\lambda}\tilde{W}$, homomorphisms between
these components in $\mathcal{O}$ and in $\mathcal{S}_{\lambda}$ coincide.
Therefore we may consider $Q^{\bullet}$ as a complex in $\mathcal{O}$.
Note, however, that the fact that $Q^{\bullet}$ is exact in 
$\mathcal{S}_{\lambda}$ only means that, as a complex in 
$\mathcal{O}$, each simple subquotient of any homology in $Q^{\bullet}$
is isomorphic to $L(x\cdot 0)$, where $x\notin {}^{\lambda}\tilde{W}$.

Next we observe that all components of $Q^{\bullet}$ are, in fact,
objects in $\underline{\mathcal{O}}$. Therefore we may apply the equivalence
$\Psi$ to obtain a complex $\Psi(Q^{\bullet})$ of the form
\begin{displaymath}
\dots\to \Psi(Q^2)\to \Psi(Q^1)\to \Psi(Q^1)\to L(w^{-1}\cdot 0)\to 0. 
\end{displaymath}
Each simple subquotient of any homology in $\Psi(Q^{\bullet})$
is isomorphic to $L(x\cdot 0)$, where $x\notin \tilde{W}^{\lambda}$.
The latter means that translation of $\Psi(Q^{\bullet})$ to the
$\lambda$-wall gives an exact complex, that is, the module 
$L(w^{-1}\cdot \lambda)$ has a BGG resolution. Therefore 
claim~\eqref{thmnew9557.2} follows from Proposition~\ref{item:kostantprop}.

To prove that claim~\eqref{thmnew9557.2} implies claim~\eqref{thmnew9557.1},
we simply reverse the above arguments. If $L(w^{-1}\cdot \lambda)$ is a 
Kostant module in the sense of Subsection~\ref{subsection:Kostant},
then $L(w^{-1}\cdot \lambda)$ has a BGG resolution
by Proposition~\ref{item:kostantprop}. Call this resolution $\hat{R}^{\bullet}$.
Consider the BGG complex $R^{\bullet}$ for $L(w^{-1}\cdot 0)$. 
Let $\tilde{R}^{\bullet}$ be the part of 
$R^{\bullet}$ indexed by elements in $X_{w^{-1}}$.
From Section~\ref{s4}, it follows that any arrow starting at an element
in $X_{w^{-1}}$ goes to an element in $X_{w^{-1}}$. This implies that 
$\tilde{R}^{\bullet}$ is a subcomplex of $R^{\bullet}$ and we also have that
$\hat{R}^{\bullet}$ is a translation of $\tilde{R}^{\bullet}$ to the $\lambda$-wall.
Since $\hat{R}^{\bullet}$ is exact, each simple subquotient of any homology 
in $\tilde{R}^{\bullet}$
is isomorphic to $L(x\cdot 0)$, where $x\notin \tilde{W}^{\lambda}$.
Now, applying $\Psi$ and then $\pi$, we get that
$\pi(\Psi(\tilde{R}^{\bullet}))$ is a BGG resolution of $\pi(L(w\cdot 0))$.
This completes the proof.
\end{proof}

Theorem~\ref{thmnew9557} extends and corrects the main result of \cite{futorny1998bgg}.

For $w\in {}^{\lambda}\tilde{W}$, let $\nabla(w)$ denote the maximal submodule
of the indecomposable injective envelope of $\pi(L(w\cdot 0))$ in 
$\mathcal{S}_{\lambda}$ which has the property that all composition 
subquotients of $\nabla(w)$ (in $\mathcal{S}_{\lambda}$) are isomorphic 
to $\pi(L(x\cdot 0))$, where $x\in {}^{\lambda}\tilde{W}$ is such that $x\geq w$. 
The module $\nabla(w)$ is the {\em costandard module} corresponding to 
$\pi(L(w\cdot 0))$ with respect to the properly stratified structure of 
$\mathcal{S}_{\lambda}$. 

\begin{corollary}\label{cor98325}
Let $w\in {}^{\lambda}\tilde{W}$ be such that $\pi(L(w\cdot 0))$ is Kostant.
Then, for all $x\in {}^{\lambda}\tilde{W}$ and
$i\in\mathbb{Z}_{\geq 0}$, we have:
\begin{displaymath}
\mathrm{Ext}_{\mathcal{S}_{\lambda}}^i(\pi(L(w\cdot 0)),\nabla(x))\cong
\begin{cases}
\mathbb{C},&\text{if }\mu^{\lambda}(w,x)\neq 0\text{ and }l(x)=l(w)+i;\\
0,& \text{otherwise} 
\end{cases}
\end{displaymath}
\end{corollary}

\begin{proof}
Since costandard modules are homologically right dual to proper standard modules,
see \cite[Theoreme~5]{Dl}, we can use \cite[Lemma~III.2.1]{happel1988triangulated} 
to compute the extension space in question in the homotopy category of complexes 
using the resolution provided by
Theorem~\ref{thmnew9557}. The claim now follows from the explicit form
of the complex \eqref{equation:singular_BGG}.
\end{proof}

\subsection{The right cell of the dominant weight}\label{s9.4}

Projective-injective modules in $\mathcal{O}_0^{\lambda}$ are indexed 
by the elements in the right Kazhdan-Lusztig cell of $w_{0}^{\lambda}w_0$,
see \cite{Ir85}. The Koszul dual of this statement is that 
simple modules of minimal Gelfand-Kirillov dimension 
in $\mathcal{O}_{\lambda}$ are indexed by the left Kazhdan-Lusztig 
cell of $w_{0}^{\lambda}$. Applying $\Psi$, we obtain that 
simple modules of minimal Gelfand-Kirillov dimension 
in $\mathcal{S}_{\lambda}$ are indexed by the right Kazhdan-Lusztig 
cell of $w_{0}^{\lambda}$. 

In many cases, see \cite{IS}, parabolic category $\mathcal{O}$ has a 
simple projective module (any such module is also injective as 
$\mathcal{O}$ has a simple preserving duality). 
For example, this is always the case in type $A$.
If a simple projective module in $\mathcal{O}^{\lambda}$ exists,
one of the projective-injective modules in the regular block is
obtained by translating this simple projective module from the wall
to the regular block. The resulting indecomposable projective-injective module
always has a predictable generalized Verma flag (just like
$P(w_0\cdot 0)$ in $\mathcal{O}_0$). By 
Proposition~\ref{item:proj_mult_free-st}, translation to 
$\mathcal{O}_{\lambda}$ of the Koszul dual simple of 
this indecomposable projective injective module has a BGG resolution.
Therefore, by Theorem~\ref{thmnew9557}, the corresponding
simple object in $\mathcal{S}_{\lambda}$ also has a BGG resolution.
These arguments prove the following statement.

\begin{corollary}\label{cornewnhstd}
Assume that $\mathcal{O}^{\lambda}$ has a simple projective module. There
is $w\in  {}^{\lambda}\tilde{W}$ in the right Kazhdan-Lusztig cell
of $w_0^{\lambda}$ such that $\pi(L(w\cdot 0))$ has a BGG resolution
in $\mathcal{S}_{\lambda}$.
\end{corollary}

This corrects the main result of \cite{futorny1998bgg} which, in the special 
case $|W_{\lambda}|=2$, claimed that $\pi(L(w_0^{\lambda}\cdot 0))$ has
a BGG resolution in $\mathcal{S}_{\lambda}$. The BGG complex constructed
in \cite{futorny1998bgg} is, indeed, a complex, cf. 
Proposition~\ref{proposition:singuar_BGG_dominant},
However, it is not always exact.
As Subsection~\ref{s7.2} shows, exactness fails, for example, in 
type $A_3$ with singularity $s_2$.

\vspace{1cm}

\noindent
V.~M.: Department of Mathematics, Uppsala University, Box. 480,
SE-75106, Uppsala, SWEDEN, email: {\tt mazor\symbol{64}math.uu.se}

\noindent
R.~M.: Department of Mathematics, Uppsala University, Box. 480,
SE-75106, Uppsala, SWEDEN, email: {\tt rafael.mrden\symbol{64}math.uu.se}\\
On leave from: Faculty of Civil Engineering, University of Zagreb, Fra 
Andrije Ka\v{c}i\'{c}a-Mio\v{s}i\'{c}a 26, 10000 Zagreb, CROATIA

\end{document}